\documentclass{article}
\usepackage[utf8]{inputenc}

\usepackage{lipsum}
\usepackage{amsfonts}
\usepackage{graphicx}
\usepackage{epstopdf}
\usepackage{multirow}
\usepackage{algpseudocode}
\ifpdf
  \DeclareGraphicsExtensions{.eps,.pdf,.png,.jpg}
\else
  \DeclareGraphicsExtensions{.eps}
\fi

\usepackage{bm}
\usepackage[dvipsnames]{xcolor}
\usepackage{pgfplots}
\usetikzlibrary{pgfplots.groupplots}
\usepackage{multirow}
\usepackage[ntheorem]{empheq}
\usepackage{cases}
\usepackage[export]{adjustbox}
\pgfplotsset{select coords between index/.style 2 args={
    x filter/.code={
        \ifnum\coordindex<#1\fi
        \ifnum\coordindex>#2\fi
    }
}}
\usepackage{algorithm}
\usepackage{geometry}
\usepackage[
	bookmarks,
	colorlinks=true, 
	pdftitle={Adaptive HROM},
	pdfauthor={},
	pdfsubject={},
	pdfkeywords={},
	urlcolor= blue,
	linkcolor= blue,
	citecolor= blue
]{hyperref}
\usepackage{cleveref}

\newtheorem{proposition}{Proposition}
\newtheorem{lemma}{Lemma}

\newtheorem{proof}{Proof}

\DeclareFontFamily{U}{mathx}{}
\DeclareFontShape{U}{mathx}{m}{n}{<-> mathx10}{}
\DeclareSymbolFont{mathx}{U}{mathx}{m}{n}
\DeclareMathAccent{\widecheck}{0}{mathx}{"71}
\newcommand{\Rbb}{\mathbb{R}}
\newcommand{\Pbb}{\mathbb{P}}

\newcommand{\Hcal}{\mathcal{H}}
\newcommand{\Ical}{\mathcal{I}}

\newcommand{\Ocal}{\mathcal{O}}

\newcommand{\Tcal}{\mathcal{T}}
\newcommand{\Tcalt}{\mathcal{T}_{\tind}}
\newcommand{\Ucal}{\mathcal{U}}
\newcommand{\Vcal}{\mathcal{V}}

\newcommand{\Vr}{\ensuremath{\mathbb{V}_r}}
\newcommand{\VrC}{\ensuremath{\overline{\mathbb{V}}_r}}
\newcommand{\Vd}[1]{\ensuremath{V_{#1}}}
\newcommand{\VNp}{\ensuremath{\Vd{\Nf}^{\np}}}
\newcommand{\Mcal}{\mathcal{M}} 

\newcommand{\Mn}{\Mcal_{\Nrh}}
\newcommand{\TM}[1]{T_{#1}\Mn} 
\newcommand{\fs}{w} 
\newcommand{\fsR}{\mathcal{W}} 
\newcommand{\rsR}{W} 
\newcommand{\redb}{A} 
\newcommand{\redC}{Z} 
\newcommand{\rup}{Q_r}

\newcommand{\tind}{\tau} 
\newcommand{\Nfh}{\ensuremath{N}} 
\newcommand{\Nf}{\ensuremath{2\Nfh}} 
\newcommand{\np}{\ensuremath{p}} 
\newcommand{\nr}{{\ensuremath{n_\tind}}} 
\newcommand{\nrp}[1]{{\ensuremath{n^{#1}_\tind}}} 
\newcommand{\nrnew}{{\ensuremath{n_{\tind+1}}}}
\newcommand{\Nr}{\ensuremath{2\nr}} 
\newcommand{\Nrh}{\ensuremath{n}} 
\newcommand{\Nrz}{{\ensuremath{n_0}}}
\newcommand{\nt}{\ensuremath{N_t}} 
\newcommand{\nd}{d} 
\newcommand{\nm}{\ensuremath{m}} 
\newcommand{\nmj}{\ensuremath{\nm}_{\tind}} 
\newcommand{\nmp}{\ensuremath{\nm}_{\upd}} 
\newcommand{\nmz}{\ensuremath{\nm_0}} 
\newcommand{\nnzJ}{s} 
\newcommand{\npUs}{\np_\deimb^*} 
\newcommand{\npAs}{\np_\redb^*} 
\newcommand{\nms}{{m_s}} 
\newcommand{\ndeimr}{r} 
\newcommand{\nrc}{r_\deimC} 
\newcommand{\pd}{\sigma} 
\newcommand{\nINL}{n_{\text{NL}}} 
\newcommand{\dt}{\Delta t} 
\newcommand{\prm}{\eta} 
\newcommand{\vprm}{\vec{\prm}}
\newcommand{\norm}[1]{\left\lVert#1\right\rVert}
\newcommand{\Adot}{\ensuremath{Y_{\redb}}} 
\newcommand{\retr}{\mathcal{R}} 
\newcommand{\Vretr}{V} 
\newcommand{\hvf}{X_\Hcal} 
\newcommand{\hrham}{\mathcal{H}^{\hr}} 
\newcommand{\hvfk}{X_{\Hcal_{\prm_k}}} 
\newcommand{\Gdec}{h} 
\newcommand{\JGdec}{D_\Gdec} 
\newcommand{\cdec}{v} 
\newcommand{\deimsn}{D} 
\newcommand{\deimP}{\mathbb{P}} 
\newcommand{\deimb}{U} 
\newcommand{\deimii}{P} 
\newcommand{\deimbn}{U^{\upd}} 
\newcommand{\deimiin}{P^{\upd}} 
\newcommand{\deimC}{C} 
\newcommand{\deimf}{\Delta} 
\newcommand{\deimR}{R} 
\newcommand{\deimsns}{\deimsn^*}
\newcommand{\deimsnc}{\widecheck{\deimsn}}
\newcommand{\deimCs}{C^*}
\newcommand{\deimCc}{\widecheck{C}}
\newcommand{\deimRs}{R^*}
\newcommand{\deimRc}{\widecheck{R}}
\newcommand{\deimsi}{S} 
\newcommand{\deimsic}{\widecheck{S}} 
\newcommand{\deimRss}{\deimsi^{\top}\deimRs}
\newcommand{\redCs}{\redC^*}
\newcommand{\hki}{\widehat{k}_1} 
\newcommand{\hkii}{\widehat{k}_2} 
\newcommand{\kii}{k_2} 
\newcommand{\alphaup}{\alpha} 
\newcommand{\betaup}{\beta} 
\newcommand{\indtheta}{\theta} 
\newcommand{\indres}{r} 
\newcommand{\inde}{e} 
\newcommand{\radapt}{C_1} 
\newcommand{\cadapt}{C_2} 
\newcommand{\ladapt}{\lambda} 
\newcommand{\dsc}{\delta} 
\newcommand{\prmaic}{\alpha} 
\newcommand{\prmbic}{\beta} 
\newcommand{\prme}{\gamma} 
\newcommand{\prms}{\Gamma} 
\newcommand{\hc}{x} 
\newcommand{\vc}{y} 
\newcommand{\fsy}{\fs} 
\newcommand{\fsq}{q_h}
\newcommand{\fsp}{p_h}
\newcommand{\dfsq}{\dot{q}_h}
\newcommand{\dfsp}{\dot{p}_h}
\newcommand{\Dbb}{\mathbb{D}} 
\newcommand{\Ecal}{\mathcal{E}} 

\DeclareMathOperator{\new}{new}
\DeclareMathOperator{\upd}{upd}
\DeclareMathOperator{\rank}{rank}
\DeclareMathOperator{\hr}{{hROM}}

\DeclareMathOperator{\argmin}{argmin}
\pgfplotsset{compat=1.18}

\newcommand{\email}[1]{\protect\href{mailto:#1}{#1}}

\title{Fully adaptive structure-preserving hyper-reduction of parametric Hamiltonian systems\thanks{Accepted in the SIAM Journal on Scientific Computing on September 2024. Submitted to the editors on September 2023.
C. Pagliantini acknowledges the MIUR Excellence Department Project awarded to the Department of Mathematics, University of Pisa, CUP I57G22000700001, and the INDAM$/$GNCS 2024 project CUP E53C23001670001.
}}

\author{Cecilia Pagliantini\thanks{Dipartimento di Matematica,
			  Universit\`a di Pisa,
			  Italy.
  (\email{cecilia.pagliantini@unipi.it}).}
\and Federico Vismara\thanks{Centre for Analysis, Scientific computing and Applications, Eindhoven University of Technology, The Netherlands. (\email{f.vismara@tue.nl}).}}

\usepackage{amsopn}

\begin{document}

\maketitle

\begin{abstract}
Model order reduction provides low-complexity high-fidelity surrogate models that allow rapid and accurate solutions of parametric differential equations. The development of reduced order models for parametric \emph{nonlinear} Hamiltonian systems is challenged by several factors: (i) the geometric structure encoding the physical properties of the dynamics; (ii) the slowly decaying Kolmogorov $n$-width of conservative dynamics; (iii) the gradient structure of the nonlinear flow velocity; (iv) high variations in the numerical rank of the state as a function of time and parameters. We propose to address these aspects via a structure-preserving adaptive approach that combines symplectic dynamical low-rank approximation with adaptive gradient-preserving hyper-reduction and parameters sampling. Additionally, we propose to vary in time the dimensions of both the reduced basis space and the hyper-reduction space by monitoring the quality of the reduced solution via an error indicator related to the projection error of the Hamiltonian vector field. The resulting adaptive hyper-reduced models preserve the geometric structure of the Hamiltonian flow, do not rely on prior information on the dynamics, and can be solved at a cost that is linear in the dimension of the full order model and linear in the number of test parameters. Numerical experiments demonstrate the improved performances of the fully adaptive models compared to the original and reduced models. 
\end{abstract}

\section{Introduction}

We are interested in the numerical solution of large-scale parametric nonlinear Hamiltonian systems that need to be tested for a large number of parameters.
Despite the success of model order reduction for time-dependent problems \cite{BGW15,HPRo22}, reduced order models for such systems face little or no computational gain, and the possible onset of numerical instabilities resulting from failing to satisfy the Hamiltonian structure of the flow.

More in details, for any given parameter $\prm\in\Rbb^{\pd}$, $\pd\geq 1$, we consider the Hamiltonian system described by the initial value problem:
find $\fs(\cdot;\prm):\Tcal:=(t^0,\infty) \rightarrow\Vd{\Nf}\subset\Rbb^{\Nf}$ such that
\begin{equation}\label{eq:HamSystem}
    \left\{
    \begin{aligned}
        & \Dot{\fs}(t;\prm)=J_{\Nf}\nabla\Hcal_{\prm}(\fs(t;\prm)), &\quad t\in\Tcal, \\
        & \fs(t^0;\prm) = \fs^0(\prm)\in\Vd{\Nf}
    \end{aligned}\right.
\end{equation}
where $\Hcal_{\prm}:\Vd{\Nf}\rightarrow \Rbb$ is the Hamiltonian of the system for any $\prm$, $\Vd{\Nf}$ is a $\Nf$-dimensional vector space,
and $J_{\Nf}$ is the so-called canonical symplectic tensor 
\begin{equation*}
J_{\Nf} :=
	\begin{pmatrix}
		 0_{\Nfh} & I_{\Nfh} \\
		-I_{\Nfh} & 0_{\Nfh} \\
	\end{pmatrix}\in\Rbb^{\Nf\times\Nf},
\end{equation*}
with $I_{\Nfh}, 0_{\Nfh} \in\Rbb^{\Nfh\times\Nfh}$ denoting the identity and zero matrices, respectively. The operator $J_{\Nf}$ identifies a symplectic structure on the phase space of the Hamiltonian system \eqref{eq:HamSystem}.
Equivalently, the vector space $\Vd{\Nf}$ can be endowed with a local basis~$\{e_i\}_{i=1}^{\Nf}$ which is symplectic and orthonormal \cite[Chapter 12]{CdS01}, that is	$e_i^\top J_{\Nf} e_j=(J_{\Nf})_{i,j}$
and $(e_i,e_j)=\delta_{i,j}$ for all $i,j=1,\ldots,\Nf$,
where $(\cdot,\cdot)$ is the Euclidean inner product.
In this coordinates system the symplectic tensor is related to the so-called symplectic 2-form $\omega$ by $\omega(u,v)=u^{\top}J_{\Nf}v$ for any $u,v\in\Vd{\Nf}$.
Throughout the paper, we also assume that, for each fixed $\prm$, the velocity field of the flow is Lipschitz continuous in the state $\fs(\cdot;\prm)$ uniformly with respect to time. 

In large-scale simulations, namely when the number $\Nf$ of degrees of freedom is large, the computational cost of solving problem \eqref{eq:HamSystem} for $\np$ instances of the parameter $\prm$ can be prohibitive. In recent years, structure-preserving reduced basis methods \cite{PM16,AH17,HP20,MN17,P19} have allowed to derive computational models that retain the symplectic structure of the phase space and can be solved efficiently in case of \emph{linear} and \emph{affine} Hamiltonian systems. However, in the presence of \emph{nonlinear} Hamiltonian systems no computational gain is obtained. In a recent work \cite{PV22} by the authors, a gradient-preserving hyper-reduction algorithm has been developed to efficiently solve nonlinear problems resulting from the symplectic model order reduction of Hamiltonian systems.
The idea is to combine a suitable decomposition of the reduced Hamiltonian with empirical interpolation (EIM) \cite{BMNP04} to derive an approximate Hamiltonian that can be evaluated at a cost independent of the size of the full order model. The method proposed in \cite{PV22} can deal, however, with only one parameter at the time. Moreover, it has been observed that
the reducibility of the nonlinear operators involved increases as the size of the reduced basis space decreases. Low-dimensional approximation spaces for conservative dynamical systems are achievable only via nonlinear model order reduction. Indeed the slowly decaying Kolmogorov $n$-width typical of transport-dominated problems and conservative dynamics \cite{deVore17} makes the use of traditional reduced basis methods based on global approximation spaces ineffective.
To address this challenge, symplectic dynamical techniques based on dynamical low-rank approximation have been developed in \cite{MN17,P19} to evolve the reduced space, but no hyper-reduction has been considered. Additionally, even in the presence of linear Hamiltonian fields, the aforementioned methods achieve a computational complexity of, at least, $O(\Nfh\np)$: for many instances of the parameters, large-scale simulations are thus simply not feasible.

This work proposes to combine symplectic dynamical model order reduction with adaptive gradient-preserving hyper-reduction and parameters sampling to yield efficient reduced order models.
The main contributions are:
\begin{enumerate}
    \item a gradient-preserving hyper-reduction based on empirical interpolation where the EIM space is updated in time via a low-rank correction that minimizes the residual. The adaptive strategy relies on a subsample of parameters chosen as to improve the approximation of the nonlinear term at the current time;
    \item a strategy to select parameter subsamples to evolve the reduced basis space and for the adaptive components of the algorithm to improve computational efficiency while ensuring favorable approximability properties;
\item an algorithm to adapt both the dimension of the reduced basis space and of the EIM hyper-reduction space to achieve the desired accuracy in situations where the (numerical) rank of the solution has high variations over time;
\item a novel error indicator to drive the adaptivity based on the error between the Hamiltonian vector field and its local projection onto the tangent space of the reduced basis space, and an algorithm to efficiently evaluate it.
\end{enumerate}
 The resulting hyper-reduced models preserve the geometric structure of the Hamiltonian flow, do not rely on prior information on the dynamics, and can be solved at a cost that is linear in the dimension $\Nfh$ of the full order model, linear in the number $\np$ of parameters and that does not depend on the product $\Nfh\np$.

The remainder of the paper is organized as follows.
In \Cref{sec:SDLR} we introduce a matrix-valued formulation of the Hamiltonian dynamical system \eqref{eq:HamSystem} and discuss its symplectic dynamical low-rank approximation.
\Cref{sec:adap_GPEIM} concerns a novel hyper-reduction technique based on empirical interpolation where the EIM basis is adapted over time via a low-rank correction and the adaptation is based on a subsample of the test parameters,
whose selection
is discussed in \Cref{sec:samples}.
Next, in \Cref{sec:rank-adaptivity} we present an algorithm to adapt the rank of the approximate reduced state and the dimension of the EIM space.
An error indicator and another subsample of the test parameters are used to drive the adaptivity, as discussed in
\Cref{sec:pAstar_selection}.
The temporal discretization of the resulting hyper-reduced dynamical system is briefly described in \Cref{sec:temp_discr_hypreduced}. After summarizing in \Cref{sec:summary} the resulting algorithm and its computational complexity, the proposed methods are tested on a set of numerical experiments in \Cref{sec:numexp}. Some concluding remarks are presented in \Cref{sec:conclusions}.

\section{Symplectic dynamical low-rank approximation}
\label{sec:SDLR}

We are interested in solving the Hamiltonian system \eqref{eq:HamSystem}
for a given set of $\np$ parameters. The problem can then be recast in a matrix-valued formulation as follows.
Let $\fsR_k:\Tcal\subset\Rbb\rightarrow\Vd{\Nf}\subset\Rbb^{\Nf}$, $1\leq k\leq\np$, denote the state variable associated with the $k$th parameter $\prm_k$, i.e. $\fsR_k(t)=\fs(t;\prm_k)$, and let $\VNp$ be the $\np$-ary Cartesian power of $\Vd{\Nf}$.
For any $t\in\Tcal$, we consider the matrix-valued quantity 
$\fsR(t)=[\fsR_1(t)|\dots|\fsR_{\np}(t)]\in\VNp\subset\Rbb^{\Nf\times\np}$ and the dynamical system:
given $\fsR^0\in \VNp$
find $\fsR\in C^1(\Tcal;\VNp)$ such that
\begin{equation}\label{eq:FOM}
\left\{\begin{aligned}
    & \dot{\fsR}(t)=\hvf(\fsR(t)):=J_{\Nf}\nabla\Hcal(\fsR(t)),\qquad t\in\Tcal, \\
    & \fsR(t^0)=\fsR^0,
\end{aligned}\right.
\end{equation}
where the Hamiltonian $\Hcal$ is the vector-valued quantity whose $k$th entry is defined as $(\Hcal(\fsR(t)))_k:=\Hcal_{\prm_k}(\fsR_k(t))$, for any $k=1,\ldots,\np$ and $t\in\Tcal$, and its gradient is the matrix $\nabla\Hcal(\fsR(t))=[\nabla_{\fsR_1}\Hcal_{\prm_1}(\fsR_1(t))|\dots|\nabla_{\fsR_{\np}}\Hcal_{\prm_{\np}}(\fsR_{\np}(t))]\in\Rbb^{\Nf\times\np}$.

For the model order reduction of \eqref{eq:FOM} we consider a nonlinear approach based on dynamical low-rank approximation \cite{KL07}. Since we need to preserve the symplectic structure of the phase space, we consider the symplectic approach proposed in \cite{P19,MN17} where, in addition, we allow the rank of the approximate state to vary in time.
Let $\Nrh:\Tcal\rightarrow\mathbb{N}_{>0}$ with
$\Nrh(t)\ll\Nfh$ for any $t\in\Tcal$. We consider a low-rank approximation $\rsR:\Tcal\rightarrow\VNp$ of the full order state $\fsR$ of the form
\begin{equation}\label{eq:LR}
    \rsR(t)=\redb(t)\redC(t),\qquad \forall t\in\Tcal,
\end{equation}
where $\redC:\Tcal\rightarrow \Rbb^{2\Nrh(t)\times\np}$ and 
$\redb: \Tcal\rightarrow\Rbb^{\Nf\times 2\Nrh(t)}$.
The dynamics of the approximate state $\rsR$ is prescribed in terms of evolution equations for $\redb$, the time-dependent reduced basis, and for $\redC$, the expansion coefficients in the reduced basis.
By imposing that the reduced space spanned by the columns of $\redb$ is a symplectic vector space at every time, the approximate reduced dynamics preserves the geometric structure of the full order model. This amounts to approximating $\fsR$ in the manifold
\begin{align*}
\Mn=\{\rsR\in\Rbb^{\Nf\times \np}:\;
    & \rsR = \redb\redC\;\mbox{with}\;
		\redb\in\Ucal_{\Nrh},\,
        \mbox{and}\\
        & \redC\in\Rbb^{2\Nrh(t)\times\np},\,\rank(\redC\redC^\top + J_{2\Nrh(t)}^\top \redC\redC^\top J_{2\Nrh(t)}) = 2\Nrh(t)\},
\end{align*}
where $\Ucal_{\Nrh}:=\{\redb\in\Rbb^{\Nf\times 2\Nrh(t)}: \redb^\top \redb=I_{2\Nrh(t)},\, \redb^\top J_{\Nf}\redb=J_{2\Nrh(t)}\}.$
%


\subsection{Reduced dynamics via Dirac--Frenkel variational principle}

To derive the evolution of the approximate state $\rsR\in\Mn$, we consider
$T_{\rsR}{\Mn}$, the tangent space of $\Mn$ at $\rsR$, that is the linear subspace of $\VNp$ containing the derivatives of all paths on $\Mn$ passing through $\rsR$.
The Dirac--Frenkel variational principle \cite{Fren34,ML64} determines the approximate trajectory $t\mapsto \rsR(t)\in\Mn$ from the condition that the vector field $\dot{\rsR}\in T_{\rsR}{\Mn}$ given by the time derivative satisfies, at every time $t$, 
\begin{equation*}
\dot{\rsR}\in\argmin_{V\in T_{\rsR}{\Mn}}\norm{V-\hvf(\rsR)}
\end{equation*}
in a suitable norm $\norm{\cdot}$.
In our case, this amounts to projecting the vector field $\hvf$ at $\rsR$ to the tangent space at $\rsR$, namely
\begin{equation}\label{eq:red}
    \dot{\rsR}=\Pi_{T_{\rsR}{\Mn}}\hvf(\rsR),
\end{equation}
where $\Pi_{T_{\rsR}{\Mn}}: \mathbb{R}^{2N\times p}\rightarrow T_{\rsR}{\Mn}$ is the orthogonal projection with respect to the symplectic $2$-form, i.e.,
$\omega\big(\hvf-\Pi_{T_{\rsR}{\Mn}}\hvf,V\big) = 0$ for any $V\in T_{\rsR}{\Mn}$, and it is defined \cite{P19,MN17} as
\begin{equation}\label{eq:proj}
    \Pi_{T_{\rsR}\Mn}(\hvf)=(I_{\Nf}-\redb\redb^\top)(\hvf\redC^\top + J_{\Nf}\hvf\redC^\top J_{2\Nrh}^\top)M^{-1}(\redC)\redC+\redb\redb^\top \hvf,
\end{equation}
where $M(\redC) := \redC\redC^\top+J_{2\Nrh}^\top \redC\redC^\top J_{2\Nrh}$.
Using the explicit expression for the projection $\Pi_{T_{\rsR}{\Mn}}$ given in \eqref{eq:proj} and the factorization \eqref{eq:LR}, the reduced dynamics \eqref{eq:red} can be equivalently re-written as
\begin{subequations}\label{eq:ROM}
\begin{empheq}[left = \empheqlbrace\,]{align}
& \dot{\redb}  = (I_{\Nf}-\redb\redb^\top)\big(J_{\Nf}\nabla\Hcal(\redb\redC)\redC^\top{+}\nabla\Hcal(\redb\redC)\redC^\top J_{2\Nrh}\big)M^{-1}(\redC), 
\label{eq:redb}\\ 
& \dot{\redC}  = J_{2\Nrh}\nabla\Hcal_r(\redC), 
\label{eq:redC}
\end{empheq}
\end{subequations}
where $\Hcal_r := \Hcal\circ\redb$. The initial condition is obtained from the projection of $\fsR(t^0)$ onto $\mathcal{M}_{\Nrh(t^0)}$.
Note that the evolution \eqref{eq:redC} of $\redC$ is a Hamiltonian flow with $2\Nrh$ degrees of freedom per parameter, and the Hamiltonian remains an invariant of motion, while the solution $A(t)$ of \eqref{eq:redb} belongs to $\Ucal_{\Nrh(t)}$ for any $t\in\Tcal$.

It can be shown \cite{P19} that, for \emph{linear} systems \eqref{eq:HamSystem}, the computational cost required to solve \eqref{eq:ROM} is linear in the dimension $\Nfh$ of the full order problem and in the number $\np$ of parameters, but it depends on the product $\Nfh\np$. Since both $\Nfh$ and $\np$ can be potentially large, solving the reduced model \eqref{eq:ROM} can still be computationally demanding. Additionally, when the Hamiltonian vector field is a \emph{nonlinear} function of the state, it is well known that 
the reduced system can be as expensive to solve as the full order one despite dimensionality reduction.

To overcome these limitations we propose: an adaptive gradient-preserving hyper-reduction strategy to efficiently deal with nonlinearities of the Hamiltonian in the evolution \eqref{eq:redC} of the coefficients $\redC$, even in the presence of a large number of parameters (\Cref{sec:adap_GPEIM});
a rank-adaptive algorithm to update the dimension of both the reduced basis space and the EIM space to optimize the approximability properties of the reduced state at a low computational cost (\Cref{sec:rank-adaptivity});
two parameter sampling strategies to efficiently perform adaptivity of the reduced basis and EIM spaces (\Cref{sec:samples} and \Cref{sec:pAstar_selection}).


\section{Adaptive gradient-preserving empirical interpolation}
\label{sec:adap_GPEIM}

Let us assume that the Hamiltonian is a nonlinear function of the state, with a general, typically non polynomial, nonlinearity.
If the Hamiltonian has also linear or quadratic components, these can be excluded from the hyper-reduction process whenever they can be evaluated at a cost that is independent of $\Nfh$.
%
The major computational bottleneck in solving equation \eqref{eq:redC} comes from the evaluation of the nonlinear reduced Hamiltonian $\Hcal_r$.
If traditional hyper-reduction techniques are applied to the full or to the reduced Hamiltonian gradient, the resulting approximate function is no longer a gradient field, leading to unstable or inaccurate approximations \cite{PM16,PV22}.
In a recent work \cite{PV22} by the authors, an adaptive gradient-preserving hyper-reduction technique has been proposed to derive efficient hyper-reduced models that retain the Hamiltonian structure of the flow.
However, this method only deals with one parameter at the time.
In this work, we extend this strategy to efficiently address the case of multiple test parameters. First, we recall the method of \cite{PV22} for $\np=1$ and then, in \Cref{sec:adaptEIM}, we consider the general case $\np\geq 2$.

Starting from the reduced order model \eqref{eq:ROM} where the reduced basis $\redb\in\Ucal_{\Nrh}$ is assumed to be given and $\redC\in\Rbb^{2\Nrh}$, \cite{PV22} proposes a decomposition of the nonlinear (part of the) reduced Hamiltonian of the form
\begin{equation}\label{eq:decomp}
    \Hcal_r(\redC) = \sum_{i=1}^\nd \cdec_i\Gdec_i(\redb\redC)=\cdec^{\top}\Gdec(\redb\redC),\qquad\mbox{for all}\; \redC\in\Rbb^{2\Nrh},
\end{equation}
where $\nd\in\mathbb{N}$, $\cdec\in\Rbb^\nd$ is a constant vector and $\Gdec:\Rbb^{\Nf}\to\Rbb^\nd$.
For hyper-reduction to be effective, the choice of the decomposition should maximize the sparsity of $h$ as a function of the reduced state $\redC$, which typically leads to
$\nd$ scaling with the full order dimension $\Nfh$. A discussion on this aspect is provided before \Cref{sec:adaptEIM} and more details can be found in \cite[Remark 3.1]{PV22}.
We point out that a sparse decomposition of the Hamiltonian is available whenever \eqref{eq:HamSystem} stems from a finite element or finite volume discretization of a Hamiltonian partial differential equation. In this case, each entry of $\Gdec$ represents the contribution to the total Hamiltonian of a single mesh element, and, as such, it only depends on the set of local degrees of freedom.
Examples of Hamiltonians that can be decomposed in such form are given in \cite{MR99} and \cite[Chapter I]{HLW06}.

Using the decomposition \eqref{eq:decomp}, the gradient of the nonlinear term $\Hcal_r$ reads
$\nabla\Hcal_r(\redC)=\redb^{\top}\JGdec^{\top}(\redb\redC)\cdec$, 
where $\JGdec:\Rbb^{\Nf}\rightarrow\Rbb^{\nd\times\Nf}$ is the Jacobian of $\Gdec$.
The next step consists in using the empirical interpolation method \cite{BMNP04} to approximate the Jacobian mapped to the reduced space, i.e., the map $\redC\in\Rbb^{2\Nrh}\mapsto\JGdec(\redb\redC)\redb\in\Rbb^{\nd\times2\Nrh}$.
%
More in details, let $\nm:\Tcal\rightarrow\mathbb{N}_{>0}$ be such that $\nm(t)\ll\nd$ for any $t\in\Tcal$.
Consider a set of $\nm=\nm(t)$ basis functions arranged as columns of the so-called EIM matrix $\deimb=\deimb(t)\in\Rbb^{\nd\times\nm(t)}$.
Let $\{i_1,\dots,i_{\nm(t)}\}\subset\{1,\dots,\nd\}$ be a set of interpolation indices that form the matrix $\deimii=\deimii(t):=[\mathbf{e}_{i_1}|\dots|\mathbf{e}_{i_{\nm(t)}}]\in\Rbb^{\nd\times\nm(t)}$ where $\mathbf{e}_i$ is the $i$th unit vector of $\mathbb{R}^{\nd}$.
Given $\deimb$ and $\deimii$, the empirical interpolation operator is defined as
\begin{equation}\label{eq:EIMinterp}
    \Ical_{\nm}\big(\JGdec(\redb\redC)\redb\big)=\deimb\deimC,\qquad
    \deimC=\deimC(t)\in\Rbb^{\nm\times2\Nrh},
\end{equation}
with $\deimC$ obtained from the interpolation constraints
$\deimii^\top \JGdec(\redb\redC)\redb=\deimii^\top\deimb\deimC$.
If $\deimii^\top\deimb$ is non-singular \cite{CS10}, the resulting empirical interpolation of the reduced Jacobian at time $t\in\Tcal$ reads, for $\redC(t)\in\Rbb^{2\Nrh(t)}$,
\begin{equation}\label{eq:DEIMproj}
\Pbb(t) \JGdec(\redb(t)\redC(t))\redb(t),
\qquad\mbox{where}\qquad
\Pbb:=\deimb(\deimii^\top\deimb)^{-1}\deimii^\top\in\Rbb^{\nd\times\nd}
\end{equation}
is the EIM projection matrix.
The EIM projection introduced in \eqref{eq:DEIMproj} can be equivalently seen as an approximation of the nonlinear term $\Hcal_r$ in \eqref{eq:decomp} by
\begin{equation}\label{eq:hhr_and_gradhhr}
    \hrham_t(\redC):=\cdec^\top\Pbb(t) \Gdec(\redb\redC)\qquad\mbox{for all}\; \redC\in\Rbb^{2\Nrh},\, t\in\Tcal.
\end{equation}
With the nonlinear term of the Hamiltonian approximated as in \eqref{eq:hhr_and_gradhhr}, the resulting hyper-reduced problem for the expansion coefficients matrix reads
\begin{equation*}
  \dot{\redC}(t)=J_{2\Nrh}\nabla \hrham_t(\redC(t)),\qquad t\in\Tcal,
\end{equation*}
with
$\nabla\hrham_t(\redC(t)) = \redb(t)^{\top}\JGdec^{\top}(\redb(t)\redC(t))\Pbb(t)^{\top}\cdec$.
This term can typically be evaluated efficiently: if $\nd=\Nfh$ and the $i$th entry of the vector-valued function $\Gdec$ only depends on, for example, the $i$th and $(i+\Nfh)$th components of the state, the Jacobian $\JGdec$ is the concatenation of two $\nd\times\nd$ diagonal matrices. This is the case for the Hamiltonian formulation of the Schr\"odinger equation that will be considered in the numerical experiments of \Cref{sec:numexp}. In this case, the gradient of the hyper-reduced Hamiltonian can be re-written as $\redb^\top \JGdec^\top(\redb\redC)\deimP^\top \cdec=\widetilde{\redb}^\top\widetilde{\nabla\Hcal}(\redb\redC)$
where $\widetilde{\redb}\in\Rbb^{2\nm\times2\Nrh}$ and $\widetilde{\nabla\Hcal}(\redb\redC)\in\Rbb^{2\nm}$ is the gradient of the full order system evaluated only at the rows corresponding to the EIM indices.
In general, the complexity of evaluating the hyper-reduced gradient $\nabla\hrham_t$ is independent of the full order dimension $\Nfh$ whenever each component of $\Gdec$ depends on $\nnzJ\ll\Nfh$ entries of the state, so that the Jacobian matrix $\JGdec$ is sparse.
This condition is typically ensured by the decomposition \eqref{eq:decomp}.

In \cite{PV22} a strategy to construct the EIM pair $(\deimb,\deimii)$ is proposed in the case where the reduced basis space is fixed over time and the problem does not depend on parameters, namely when the Hamiltonian system needs to be solved only once.
When applied to the parametric case, such strategy requires the construction and update of one EIM pair per parameter. In the presence of a high number of parameters this option is not computationally affordable. In the next section we discuss how to address the parametric hyper-reduction with a gradient-preserving and parameter-efficient adaptive strategy.

\subsection{Adaptivity of the EIM basis and interpolation points}
\label{sec:adaptEIM}

Unlike global reduced basis methods, dynamical low-rank approximation is not data-driven, i.e., it does not rely on an offline phase to build the reduced space. 
It is then necessary to adapt the EIM space and interpolation indices during the simulation to achieve sufficiently accurate approximations \cite{Peher15,Peher20,NB23}. In \cite{Peher20}, extended from \cite{Peher15}, the EIM basis is adapted over time via a low-rank correction that acts at the sampling points which minimize the DEIM residual.
These works only deal with the case where the system has to be solved for one test parameter at a time and the reduced basis space is fixed. In this section, we develop an adaptive hyper-reduction algorithm to efficiently address the situation where an efficient and accurate approximation has to be provided at all times for a large number of test parameters.

Let us assume, for the time being, that the EIM pair $(\deimb,\deimii)\in\Rbb^{\nd\times\nm}\times\Rbb^{\nd\times\nm}$ is given and that the reduced basis $\redb\in\Ucal_{\Nrh}$ is fixed. In this section we propose a strategy to update the EIM pair to the quantity $(\deimbn,\deimiin)\in\Rbb^{\nd\times\nmp}\times\Rbb^{\nd\times\nmp}$ so that the gradient structure of the Hamiltonian vector field is preserved and the updated EIM space provides a better approximation of the nonlinear Hamiltonian vector field for all test parameters.
Let us consider the matrix-valued function $(\redb,\redC)\in\Ucal_{\Nrh}\times \Rbb^{2\Nrh\times\np}\mapsto\deimsn(\redb,\redC)=\deimsn=[\deimsn^{(1)}|\dots|\deimsn^{(\np)}]\in\Rbb^{\nd\times 2\Nrh\np}$, where $\deimsn^{(k)}$ denotes the reduced Jacobian associated to the $k$th parameter, namely
\begin{equation}\label{eq:Ni}
    \deimsn^{(k)}=\JGdec(\redb\redC_k)\redb
    \in\Rbb^{\nd\times 2\Nrh} \qquad k=1,\dots,\np,
\end{equation}
with $\redC_k$ the $k$th column of the matrix $\redC\in\Rbb^{2\Nrh\times\np}$.
The EIM approximation of $\deimsn$ with the current EIM basis $\deimb$ and interpolation indices $\deimii$ is given by $\deimb\deimC$ as in \eqref{eq:EIMinterp}, with coefficient matrix
\begin{equation}\label{eq:Cmatrix}
    \deimC=(\deimii^\top \deimb)^{-1}\deimii^\top\deimsn
    = [\deimC^{(1)}|\dots|\deimC^{(\np)}]\in\Rbb^{\nm\times 2\Nrh\np},
\end{equation}
and $\deimC^{(k)}:=(\deimii^\top \deimb)^{-1}\deimii^\top \deimsn^{(k)}\in\Rbb^{\nm\times 2\Nrh}$.
We define the EIM residual as the quantity $\deimR:=\deimb\deimC-\deimsn\in\Rbb^{\nd\times 2\Nrh\np}$.

Two quantities play a crucial role in the proposed EIM update: a parameter subsample and a set of sample indices.
Let us consider a subset of $\npUs\leq\np$ sample parameters $\{\prm_{i_1},\dots,\prm_{i_{\npUs}}\}\subset\{\prm_1,\dots,\prm_\np\}$ and define the matrices
\begin{equation*}
    \deimsns=[\deimsn^{(i_1)}|\dots|\deimsn^{(i_{\npUs})}]\in\Rbb^{\nd\times 2\Nrh\npUs} \quad\mbox{and}\quad
    \deimsnc=[\deimsn^{(i_{\npUs+1})}|\dots|\deimsn^{(i_\np)}]\in\Rbb^{\nd\times 2\Nrh(\np-\npUs)}.
\end{equation*}
Analogously one can define
$\deimCs=(\deimii^\top \deimb)^{-1}\deimii^\top\deimsns\in\Rbb^{\nm\times 2\Nrh\npUs}$, $\deimCc=(\deimii^\top \deimb)^{-1}\deimii^\top\deimsnc\in\Rbb^{\nm\times 2\Nrh(\np-\npUs)}$, and
$\deimRs:=\deimb\deimCs-\deimsns\in\Rbb^{\nd\times 2\Nrh\npUs}$. 
We also consider a set of $\nms\geq\nm$ sample indices $\{i_1,\dots,i_{\nms}\}\subset\{1,\dots,\nd\}$ and introduce the matrices
\begin{equation}\label{eq:Smatrix}
    \deimsi=[\mathbf{e}_{i_1}|\dots|\mathbf{e}_{i_{\nms}}]\in\Rbb^{\nd\times\nms}
    \quad\mbox{and}\quad
    \deimsic=[\mathbf{e}_{i_{\nms+1}}|\dots|\mathbf{e}_{i_\nd}]\in\Rbb^{\nd\times(\nd-\nms)}.
\end{equation}
The updated EIM basis $\deimbn$ is obtained as a rank-$\ndeimr$ correction of the rows of $\deimb$ corresponding to the sample indices, that is
\begin{equation*}
    \deimsi^\top \deimbn = \deimsi^\top \deimb+\rup,
    \qquad \deimsic^\top \deimbn=\deimsic^\top \deimb,
\end{equation*}
where $\rup\in\Rbb^{\nms\times\nm}$ has rank $\ndeimr$. Note that this trivially implies that $\ndeimr\leq\min(\nms,\nm)=\nm$. The update $\rup$ is chosen to minimize the quantity
\begin{equation*}
    \sum_{\ell=1}^{\npUs}\lVert \deimsi^\top(\deimbn\deimC^{(i_\ell)}-\deimsn^{(i_\ell)})\rVert_F^2=\lVert \deimsi^\top(\deimbn\deimCs-\deimsns)\rVert_F^2=\lVert \deimsi^\top \deimRs+\rup \deimCs\rVert_F^2.
\end{equation*}
In other words, the rank-$\ndeimr$ update $\rup$ minimizes, at the sample indices $\deimsi$, the EIM residual associated to the sample parameters $\{\prm_{i_1},\dots,\prm_{\npUs}\}$. Clearly, the choice of the sample indices and parameters is crucial and will be discussed in \Cref{sec:samples}.


In \Cref{prop:DEIM_update} we derive the update $\rup$ by solving the corresponding minimization problem, for any $1\leq\ndeimr\leq\rank(\deimCs)$. First, we recast the problem in an equivalent form as described in the following result.

\begin{lemma}\label{lemma:min_prob_equivalence_v3}
    Let $\Ocal(\alphaup,\betaup):=\lVert M_1+\alphaup\betaup^\top M_2\rVert_F^2$, with given $M_1\in\Rbb^{q_1\times g}$ and $M_2\in\Rbb^{q_2\times g}$. Then, for $1\leq\ndeimr\leq q_2$, it holds
    \begin{equation*}
        \min_{(\alphaup,\betaup)\in\Vr} \Ocal(\alphaup,\betaup)=\min_{(\alphaup,\betaup)\in\VrC} \Ocal(\alphaup,\betaup),
    \end{equation*}
    where 
    \begin{equation*}
        \begin{aligned}
        & \Vr:=\{(\alphaup,\betaup)\in\Rbb^{q_1\times \ndeimr}\times\Rbb^{q_2\times\ndeimr}, \rank(\alphaup)=\rank(\betaup)=\ndeimr\},\\
        & \VrC:=\{(\alphaup,\betaup)\in\Rbb^{q_1\times \ndeimr}\times\Rbb^{q_2\times\ndeimr}, \rank(\alphaup)=\rank(\betaup)=\ndeimr,\,\alphaup^\top\alphaup\mbox{ diagonal}\}.
        \end{aligned}
    \end{equation*} 
\end{lemma}
\begin{proof}
Let $m_1:=\min_{(\alphaup,\betaup)\in\Vr}\Ocal(\alphaup,\betaup)$ and $m_2:=\min_{(\alphaup,\betaup)\in\VrC}\Ocal(\alphaup,\betaup)$.
Clearly, $m_1\leq m_2$ since $\VrC\subset\Vr$.

Let $(\overline{\alphaup},\overline{\betaup})\in\Vr$ be such that $\Ocal(\overline{\alphaup},\overline{\betaup})=m_1$.
Let $\overline{Q}$ be the matrix with orthogonal columns and $\overline{W}$ the upper triangular with ones on the main diagonal obtained from the QR factorization of $\overline{\alphaup}$ without normalization. Define $\overline{Z}:=\overline{\betaup}\overline{W}^\top$. Then $(\overline{Q},\overline{Z})\in\VrC$ and $\Ocal(\overline{Q},\overline{Z})=\Ocal(\overline{\alphaup},\overline{\betaup})=m_1$, which implies that $m_2\leq \Ocal(\overline{Q},\overline{Z})=m_1$.
Hence, we can conclude that $m_1=m_2$.
\end{proof}

\begin{proposition}\label{prop:DEIM_update}
    Assume we are given $\deimsi\in\Rbb^{\nd\times\nms}$, $\deimRs\in\Rbb^{\nd\times2\Nrh\npUs}$, and $\deimCs\in\Rbb^{\nm\times 2\Nrh\npUs}$ with $\nrc:=\rank(\deimCs)\leq\nm$. Let $\Ucal^*\Sigma^*(\Vcal^*)^\top$ be the compact singular value decomposition of $\deimCs$, with $\Ucal^*\in\Rbb^{\nm\times\nrc}$, $\Sigma^*\in\Rbb^{\nrc\times\nrc}$, $\Vcal^*\in\Rbb^{\Nrh\npUs\times\nrc}$. Let $\Ocal(\alphaup,\betaup):=\lVert\deimRss+\alphaup\betaup^\top \deimCs\rVert_F^2$.
    Then, for any $1\leq\ndeimr\leq\rank(\deimRss\Vcal^*)\leq\nrc$, the solution of the minimization problem $\min_{(\alphaup,\betaup)\in\Vr}\Ocal(\alphaup,\betaup)$ is
    \begin{equation*}
        \alphaup\betaup^\top = -(\deimRss\Vcal^*)_\ndeimr(\Sigma^*)^{-1}(\Ucal^*)^\top,
    \end{equation*}
    where $(\deimRss\Vcal^*)_\ndeimr$ denotes the $\ndeimr$-truncated SVD of $\deimRss\Vcal^*$. Moreover, it holds
    \begin{equation*}
        \min_{(\alphaup,\betaup)\in\Vr}\Ocal(\alphaup,\betaup)=\lVert \deimRss\rVert_F^2 - \sum_{i=1}^\ndeimr\overline{\sigma}_i^2,
    \end{equation*}
    where $\overline{\sigma}_1\geq\dots\geq\overline{\sigma}_\ndeimr$ are the $\ndeimr$ dominant singular values of $\deimRss\Vcal^*\in\Rbb^{\nm\times \nrc}$. 
    In particular, if $\ndeimr=\rank(\deimRss\Vcal^*)$, then $\alphaup\betaup^\top=-\deimRss(\deimCs)^{\dagger}$ and
    \begin{equation}\label{eq:min_DEIMupdate}
        \min_{(\alphaup,\betaup)\in\Vr}\Ocal(\alphaup,\betaup)=\lVert\deimRss\rVert_F^2 - \lVert \deimRss\Vcal^*\rVert_F^2.
    \end{equation}
\end{proposition}
\begin{proof}
    Let $\ndeimr\in\{1,\dots,\nm\}$ be fixed and let $\alphaup_i\in\Rbb^{\nms}$ and $\betaup_i\in\Rbb^{\nm}$, for $i=1,\dots,\ndeimr$, denote the column vectors of $\alphaup$ and $\betaup$, respectively, so that $\alphaup=[\alphaup_1|\dots|\alphaup_\ndeimr]\in\Rbb^{\nms\times\ndeimr}$ and $\betaup=[\betaup_1|\dots|\betaup_\ndeimr]\in\Rbb^{\nm\times\ndeimr}$. Owing to \Cref{lemma:min_prob_equivalence_v3}, we look for $(\alphaup,\betaup)$ in $\VrC$. Using the definition of the Frobenius norm and the linearity and cyclic properties of the trace, together with the fact that the columns of $\alphaup$ are orthogonal, the cost function can be rewritten as
    \begin{equation}\label{eq:cost_function_expanded}
        \Ocal(\alphaup,\betaup) 
         =\lVert \deimRss\rVert_F^2+\sum_{i=1}^\ndeimr\lVert\alphaup_i\rVert^2\lVert(\deimCs)^\top\betaup_i\rVert^2+2\sum_{i=1}^\ndeimr\betaup_i^\top \deimCs(\deimRss)^\top\alphaup_i.
    \end{equation}
    Defining $\gamma:=\Sigma^*(\Ucal^*)^\top\betaup\in\Rbb^{\nrc\times\ndeimr}$ we find
    \begin{equation*}
        \Ocal(\alphaup,\betaup)=\widehat{\Ocal}(\alphaup,\gamma):=\lVert \deimRss\rVert_F^2+\sum_{i=1}^\ndeimr\lVert\alphaup_i\rVert^2\lVert\gamma_i\rVert^2+2\sum_{i=1}^\ndeimr\gamma_i^\top(\deimRss\Vcal^*)^\top\alphaup_i.
    \end{equation*}
    Computing the partial derivatives of $\widehat{\Ocal}$ with respect to $\alphaup_i$ and $\gamma_i$, for $i=1,\dots,\ndeimr$, and setting them equal to zero we obtain
    \begin{equation}\label{eq:alphai_beta_i}
        \gamma_i=-\displaystyle\frac{(\deimRss\Vcal^*)^\top\alphaup_i}{\lVert\alphaup_i\rVert^2}, \qquad\deimRss\Vcal^*(\deimRss\Vcal^*)^\top\alphaup_i=\sigma_i^2\alphaup_i,
    \end{equation}
    where $\sigma_i:=\lVert\alphaup_i\rVert^{-1}\lVert(\deimRss\Vcal^*)^\top\alphaup_i\rVert$ and $\alphaup_i\neq0$ for all $i$, otherwise $\alphaup$ would not have rank $\ndeimr$. This shows that $(\alphaup,\betaup)$ is a stationary point of $\Ocal$ if and only if the columns of $\alphaup$ are left singular vectors of $\deimRss\Vcal^*$ and the columns of $\betaup$ are given by $\betaup_i=\Ucal^*(\Sigma^*)^{-1}\gamma_i$. To determine which choice of $\alphaup_i$ corresponds to the global minimum of $\Ocal$, we observe that, by inserting the expression for $\gamma_i$ from \eqref{eq:alphai_beta_i} into \eqref{eq:cost_function_expanded}, we get
    \begin{equation*}
        \Ocal(\alphaup,\betaup)=
        \lVert \deimRss\rVert_F^2-\sum_{i=1}^\ndeimr\sigma_i^2.
    \end{equation*}
    Therefore, we select $\alphaup_i$ as left singular vectors of $\deimRss\Vcal^*$ corresponding to the $\ndeimr$ largest singular values. Moreover, we may assume that $\ndeimr\leq\rank(\deimRss\Vcal^*)$, since including singular vectors corresponding to zero singular values does not affect the cost function. In other words, if $\deimRss\Vcal^*=\overline{\Ucal}\overline{\Sigma}\overline{\Vcal}^\top$ is the singular value decomposition of $\deimRss\Vcal^*$, we set $\alphaup_i=\overline{u}_i$ for $1\leq i\leq \ndeimr$, which gives $\sigma_i=\overline{\sigma}_i$ and $\gamma_i=-\overline{\sigma}_i\overline{v}_i$.
    The optimal update $\alphaup\betaup^\top$ is then
    \[ \sum_{i=1}^\ndeimr\alphaup_i\gamma_i^\top(\Sigma^*)^{-1}(\Ucal^*)^\top
    =-\sum_{i=1}^\ndeimr\overline{\sigma}_i\overline{u}_i\overline{v}_i^\top(\Sigma^*)^{-1}(\Ucal^*)^\top=-(\deimRss\Vcal^*)_\ndeimr(\Sigma^*)^{-1}(\Ucal^*)^\top.
    \]
\end{proof}
In this work, we focus on full-rank updates, namely $\ndeimr=\nrc$, which yields
$\deimsi^\top\deimbn=\deimsi^\top\deimb-\deimsi^\top\deimRs(\deimCs)^{\dagger}$.
After $\deimb$ has been updated to $\deimbn$, the interpolation indices $\deimii$ are also modified accordingly: one may recompute all interpolation indices by applying a greedy algorithm \cite{CS10} or QDEIM \cite{DG16}.
In this work we use QDEIM, which is based on a QR factorization of $(\deimbn)^\top$ with column pivoting, since it enjoys favorable theoretical properties and it is computationally efficient in practice.
Another option is to only replace the indices corresponding to the basis vectors that have undergone the largest rotations in the EIM basis update, as suggested in \cite[Section 4.1]{Peher15}.

\subsection{Selection of sample indices and sample parameters for hyper-reduction}\label{sec:samples}

The update of the EIM basis provided in \Cref{prop:DEIM_update} ensures that, if $\npUs=\np$, the updated EIM basis improves the approximation of the nonlinear term at the current time by lowering the residual compared to the original EIM basis. 
This result is in agreement with the single-parameter case \cite[Theorem 4.1]{PV22}. In fact, using \eqref{eq:min_DEIMupdate} with $\npUs=\np$, the residual associated with the updated EIM basis satisfies
\begin{equation}\label{eq:res_star}
\begin{aligned}
    \lVert \deimbn\deimC-\deimsn\rVert_F^2
    & = \lVert \deimsi^\top(\deimbn\deimC-\deimsn)\rVert_F^2+\lVert\deimsic^\top(\deimbn\deimC-\deimsn)\rVert_F^2\\
    & = \lVert \deimsi^\top \deimR\rVert_F^2-\lVert \deimsi^\top \deimR\Vcal\rVert_F^2+\lVert\deimsic^\top(\deimb\deimC-\deimsn)\rVert_F^2\\
    & = \lVert \deimb\deimC-\deimsn\rVert_F^2-\lVert \deimsi^\top \deimR\Vcal\rVert_F^2<\lVert \deimb\deimC-\deimsn\rVert_F^2.
\end{aligned}
\end{equation}
If $\npUs<\np$ instead, the update ensures an improved approximation for the sample parameters but, in general, not for all parameters.
The next result gives a constructive way to select $\npUs$ sample parameters such that this requirement is satisfied. 
\begin{proposition}\label{prop:delta}
Assume that $(\deimb,\deimii)$ is a given EIM pair and that $\deimsi$ is a given matrix of sample indices.
Let $\deimsn$ be a matrix of nonlinear terms, as defined in \eqref{eq:Ni}, $\deimC:=(\deimii^\top \deimb)^{-1}\deimii^\top\deimsn$ and $\deimR:=\deimb\deimC-\deimsn$.
Let us denote with a superscript $*$ the aforementioned quantities when restricted to a given set of $\npUs\leq\np$ sample parameters.
If $\deimbn$ is the update of the EIM basis $\deimb$ obtained with $\npUs$ parameters and $\ndeimr=\nrc$, as described in \Cref{prop:DEIM_update}, then $\lVert \deimbn\deimC-\deimsn\rVert_F<\lVert \deimb\deimC-\deimsn\rVert_F$ if and only if 
    \begin{equation}\label{eq:delta}
        \dsc:=\lVert \deimsi^\top \deimR\rVert_F^2-\lVert \deimsi^\top (\deimR-\deimRs(\deimCs)^{\dagger}\deimC)\rVert_F^2>0.
    \end{equation}
\end{proposition}

\begin{proof}
    The residual associated with the unused $\np-\npUs$ parameters is
    \begin{align*}
        \lVert \deimbn\deimCc-\deimsnc\rVert_F^2
        & = \lVert \deimsi^\top(\deimbn\deimCc-\deimsnc)\rVert_F^2+\lVert \deimsic^\top (\deimbn\deimCc-\deimsnc)\rVert_F^2\\
        & = \lVert \deimsi^\top \deimb\deimCc-\deimsi^\top \deimRs(\deimCs)^{\dagger}\deimCc-\deimsi^\top\deimsnc\rVert_F^2+\lVert\deimsic^\top\deimRc\rVert_F^2\\
        & = \lVert \deimsi^\top\deimRc\rVert_F^2+\lVert \deimsi^\top \deimRs(\deimCs)^{\dagger}\deimCc\rVert_F^2-2(\deimsi^\top\deimRc,\deimsi^\top \deimRs(\deimCs)^{\dagger}\deimCc)_F+\lVert \deimsic^\top\deimRc\rVert_F^2\\
        & = \lVert \deimb\deimCc-\deimsnc\rVert_F^2+\lVert \deimsi^\top\deimRc-\deimsi^\top \deimRs(\deimCs)^{\dagger}\deimCc\rVert_F^2-\lVert\deimsi^\top\deimRc\rVert_F^2.
    \end{align*}
    Note that $\lVert \deimsi^\top\deimRs-\deimsi^\top \deimRs(\deimCs)^{\dagger}\deimCs\rVert_F^2=\lVert\deimsi^\top\deimRs\rVert_F^2-\lVert\deimsi^\top\deimRs\Vcal^*\rVert_F^2$, 
    so that $\lVert \deimsi^\top\deimRc-\deimsi^\top \deimRs(\deimCs)^{\dagger}\deimCc\rVert_F^2= \lVert \deimsi^\top\deimR-\deimsi^\top \deimRs(\deimCs)^{\dagger}\deimC\rVert_F^2-\lVert\deimsi^\top\deimRs\rVert_F^2+\lVert\deimsi^\top\deimRs\Vcal^*\rVert_F^2$ and
    \begin{equation}\label{eq:res_check}
        \lVert \deimbn\deimCc-\deimsnc\rVert_F^2=\lVert \deimb\deimCc-\deimsnc\rVert_F^2-\dsc+\lVert\deimsi^\top\deimRs\Vcal^*\rVert_F^2.
    \end{equation}
    The conclusion follows by summing identity \eqref{eq:res_star} with $\ast$ quantities and \eqref{eq:res_check}, which gives $\lVert\deimbn\deimC-\deimsn\rVert_F^2=\lVert\deimb\deimC-\deimsn\rVert_F^2-\dsc$.
\end{proof}
The condition $\dsc>0$ is trivially satisfied when $\npUs=\np$, since $\deimRs=\deimR$, $\deimCs=\deimC$ and $\deimC^{\dagger}\deimC=\Vcal\Vcal^\top$ so that equation \eqref{eq:delta} yields $\dsc=\lVert\deimsi^\top\deimR\Vcal\rVert_F$.
The result of \Cref{prop:delta} suggests a strategy to select both the sample indices $\deimsi$  and the $\npUs$ sample parameters simultaneously.
We propose an iterative strategy in which we add one parameter per iteration, as
summarized in \Cref{alg:DEIM_update}.
We select $\npUs$ parameters via QR factorization of $\redC$ with column pivoting and as sample indices those associated with the rows of $\deimRs\Vcal^*$ of maximum norm.
By maximizing $\lVert \deimsi^\top \deimRs\Vcal^*\rVert_F$, we have that the updated EIM residual associated to the given set of $p_U^*$ sample parameters is minimized with respect to the EIM residual obtained with the old EIM basis, as prescribed by \eqref{eq:res_star}.

Once the matrix $\deimsi$ of sample indices has been constructed, we check whether $\dsc$ is positive. 
If this is the case the algorithm stops,
otherwise a new sample parameter is added to the set, and the sample indices are recomputed.
Once the sample indices $S$ and the $\npUs$ parameters have been selected, the EIM basis is updated with the low-rank correction $-\deimsi^\top\deimRs(\deimCs)^{\dagger}$ as given in \Cref{prop:DEIM_update}. The correction is already available from the computation of $\dsc$ from the previous steps.
Note that the evaluation of $\dsc$ requires the knowledge of the EIM basis $\deimb$ and of the nonlinear matrix $\deimsn$ only at the interpolation and sample indices, which implies
that the selection algorithm is efficient when both $\nm$ and $\nms$ are much smaller than $\nd$.
%
\begin{algorithm}[H]
    \caption{Selection of sample indices and parameters and EIM update}\label{alg:DEIM_update}
    \begin{algorithmic}[1]
        \Procedure{$(\deimbn,\deimiin)$\,=\,EIM\_upd}{$\redb$, $\redC$, $\deimb$, $\deimii$, $\tau_{\nms}$}
        \State $\npUs\gets 0$,\, $\dsc\gets -1$,\, $\widetilde{\redC}\gets\redC$, $\redC^*=[\;]$
        \While {$\dsc<0$}
        \State $\npUs\gets \npUs+1$
        \State Compute one parameter $\prm_{i_{\npUs}}$ via pivoted QR of $\widetilde{\redC}$
        \State Update $\widetilde{\redC}$ by removing the column associated with the pivot
        \State Update $\redC^*=[\redC^*|Z_{i_{\npUs}}]$
        \State Compute $\deimsns=[\deimsn^{(i_1)}|\dots|\deimsn^{(i_{\npUs})}]\in\Rbb^{\nd\times2\Nrh\npUs}$ from $(\redb,\redCs)$ as in \eqref{eq:Ni}
        \State Compute $\deimCs=(\deimii^\top \deimb)^{-1}\deimii^\top \deimsns\in\Rbb^{\nm\times2\Nrh\npUs}$
        \State $[\Ucal^*,\Sigma^*,\Vcal^*]=\text{SVD}(\deimCs)$
        \State $\deimRs\Vcal^*=\deimb\Ucal^*\Sigma^*-\deimsns\Vcal^*$
        \State $\deimsi=v(v>\tau_{\nms})$ with $v=\text{vecnorm}((\deimRs\Vcal^*)^\top)\in\Rbb^\nd$
        \State Compute $\deimsi^\top \deimRs(\deimCs)^{\dagger}= \deimsi^\top \deimRs\Vcal^*{\Sigma^*}^{-1}{\Ucal^*}^\top$
        \State $\dsc\gets\lVert \deimsi^\top \deimR\rVert_F^2-\lVert \deimsi^\top \deimR-\deimsi^\top \deimRs(\deimCs)^{\dagger}\deimC\rVert_F^2$
        \EndWhile
        \State $\deimsi^\top \deimbn\gets \deimsi^\top \deimb-\deimsi^\top \deimRs(\deimCs)^{\dagger}$ and
        $\deimsic^\top \deimbn\gets \deimsic^\top \deimb$
        \State Orthogonalize $\deimbn$
        \State $\deimiin\gets \textsc{QDEIM}(\deimbn)$
        \EndProcedure
    \end{algorithmic}
\end{algorithm}
The arithmetic complexity of \Cref{alg:DEIM_update} is $O(\Nfh\npUs\Nrh)+O(\nd\npUs\Nrh\nrc)+O(\nd\nm^2)+O(\np\npUs\Nrh)+O(\np\Nrh\nm\nms)$.
Here, we assume that the number of non-zero entries in each row of $\JGdec$ is $\nnzJ\ll\nd$ and that $\nnzJ\leq\nrc\leq\nm$. Recall that $\nm\leq\nms\leq \nd$, so that we also have $\nnzJ\leq\nms$.
A detailed analysis of the costs is as follows.
The pivoted QR of $\redC\in\Rbb^{2\Nrh\times\np}$ truncated at $\npUs$ pivots has complexity $O(\np\npUs\Nrh)$.
Next, to compute the nonlinear matrix $\deimsns$ defined in \eqref{eq:Ni}, the reduced Jacobian $\JGdec(\redb\redCs)\redb$ has to be evaluated at the columns of the matrix $\redC$ associated with $\npUs$ sample parameters:
$O(\Nfh\npUs\Nrh)$ operations are required to reconstruct the reduced order solution at the $\npUs$ sample parameters and
$O(\nd\npUs\Nrh\nnzJ)$ operations are needed for the multiplication of the Jacobi matrix with the reduced basis $\redb$.
%
%
The computation of $\deimCs$ and its SVD require $O(\nm^3)+O(\nm^2\Nrh\npUs)$ and
$O(\nm\Nrh\npUs\min{(\nm,2\Nrh\npUs)})$ operations, respectively.
%
The matrix-matrix multiplications involved in the construction of $\deimRs\Vcal^*$ have complexity $O(\nd\nm\nrc)+O(\nd\Nrh\nrc\npUs)$.
%
For the evaluation of $\dsc$ defined in \eqref{eq:delta}, 
%
we compute $\deimsi^\top \deimRs\Vcal^*(\Sigma^*)^{-1}(\Ucal^*)^\top \deimC$ with a complexity of
$O(\nms\nrc\nm)$, while
%
$\deimsi^\top \deimR$ can be computed in $O(\nms\nm^2)+O(\nms\nm\Nrh\np)$ operations.
Then, at line 12, the multiplication by the matrix $\deimC$ has complexity $O(\np\Nrh\nm\nms)$.
%
Finally, the update is performed at line 14 (complexity $O(\nms\nm)$), the columns of the resulting matrix are re-orthogonalized ($O(\nd\nm^2)$) and the EIM interpolation indices are recomputed (also $O(\nd\nm^2)$ with QDEIM).

We point out that, whenever the solution of the problem is not localized in space, the number $\nms$ of sampling points for the EIM basis update might become proportional to $\nd$, as shown in the numerical experiment of \Cref{sec:exp_nonlocalized}. In such situation, one might take $\npUs=\npAs\geq 2\Nrh$ parameters for the EIM update via pivoted QR of $Z$ without explicitly forming $\delta$. Although it has been numerically observed that this choice provides a sufficient condition for the positivity of $\delta$, the greedy strategy in \Cref{alg:DEIM_update} provides the minimum number $\npUs$ of parameters to have $\delta>0$ and it is still computationally efficient even when $\nms$ is close to $\nd$ (see \Cref{fig:nonlocalized_times_err_ms}). We postpone further discussion to \Cref{sec:numexp}.


\section{Adaptivity of the approximation rank and of the EIM space dimension}\label{sec:rank-adaptivity}
The accuracy of the approximation introduced by model order reduction and hyper-reduction and the efficiency of the corresponding algorithms are determined by the dimension of the reduced basis space and of the EIM space.
%
%
It is indeed desirable that both dimensions are significantly smaller than that of the full order model to ensure a computational gain. On the other hand, the choice of the reduced space dimension should be strictly related to the numerical rank of the full solution: a too small reduced space might yield an inaccurate representation of the full order dynamics, while a too large value might lead to over-approximation and little or no runtime gain.
Moreover, the numerical rank of the full order solution is likely to change over time depending on the variability of the solution as a function of the parameter and on the nonlinearity of the dynamics.

We propose an algorithm to update the dimension of both the reduced space and of the EIM space at the beginning of each temporal sub-interval introduced by the temporal discretization of the dynamics.
Let us consider the splitting of the time domain $\Tcal$ into the
union of intervals $\Tcalt:=(t^{\tind},t^{\tind+1}]$, $\tind=0,\ldots,\nt-1$, with $t^{\nt}:=T$, and let the local time step be defined as $\dt_{\tind}=t^{\tind+1}-t^{\tind}$ for every $\tind$.
We assume that the dimensions of the reduced and EIM space can change only at the discrete time instants $\{t^{\tind}\}_{\tind=1}^{\nt-1}$. In each temporal subinterval $\Tcalt$, the dimension of the reduced space is denoted by $\Nr$ while the dimension of the EIM space is $\nmj$, the approximate reduced state at time $t^{\tind}$ is $\rsR_{\tind}$ and its low-rank factors are
$\redb_{\tind}$ and $\redC_{\tind}$.

\subsection{Rank adaptivity}

At the initial time, the rank $\Nrz$ of the approximate state $\rsR_0$ is determined in such a way that $\rsR_0$ approximates the exact initial condition with a chosen accuracy.
As time evolves, an error indicator is used to check the quality of the current approximation and, based on a certain criterion, the reduced basis space is enlarged or reduced.
The choice of a suitable error indicator is a crucial feature of the rank adaptivity algorithm: it should provide an accurate indication of the current approximation error with respect to the full order solution and, at the same time, it should be computed at a relatively low computational cost.

Rank-adaptation has received considerable attention in recent years. A family of works proposes adaptive techniques based on the spectral properties of the approximate solution: a recompression of the approximate tensor decomposition of the solution is proposed in \cite{EL17} for the numerical solution of the Vlasov-Poisson equation; in \cite{CKL22} the rank of the approximation is doubled at each time step based on past information of the solution and then truncated SVD is applied; in \cite{HNS23} an additional singular value is propagated in time and used as an indicator to select the rank of the solution at the new time step. The aforementioned methods are relatively cheap in terms of computational time but do not provide any indication on the error with respect to the full order solution. Moreover, they do not consider parameter-dependent problems and cannot ensure that the geometric structure of the problems at hand is preserved.
A rank-adaptive, structure-preserving reduced basis method was introduced in \cite{HPR22}, with an error indicator based on the linearized residual of the full order model. Although this quantity provides an accurate estimate of the error between the reduced and the full order solution, its computation requires to solve a linear system of dimension $\Nf$ for all test parameters.
In many practical situations this is simply not affordable.
%
Another family of works \cite{CL23,GA22} proposes to consider the angle between the velocity of the flow, the Hamiltonian vector field $\hvf=J_{\Nf}\nabla\Hcal$ in our case, and its projection onto the tangent space to $\Mn$. The rationale behind this choice is that, if the reduced space is able to accurately capture the full dynamics, then the full order Hamiltonian vector field is ``almost parallel" to its projection.
With our notation, the angle at $\rsR\in\Mn$ is defined as
\begin{equation}\label{eq:theta_errind}
    \indtheta(\rsR) = \arccos\left(\frac{\lVert\Pi_{\TM{\rsR}}\hvf(\rsR)\rVert_F}{\lVert\hvf(\rsR)\rVert_F}\right),
\end{equation}
and we let $\indtheta_{\tind}:=\indtheta(\rsR_{\tind})$ denote the value obtained by replacing the approximate state $\rsR$ with its time discretization at time $t^{\tind}$.
This quantity can be evaluated at a relatively low computational cost since $\Pi_{\TM{\rsR_{\tind}}}\hvf(\rsR_\tind)=\Adot(\redb_\tind,\redC_\tind)\redC_\tind+\redb_\tind\redb_\tind^\top \hvf(\rsR_\tind)$
where $\rsR_\tind=\redb_\tind\redC_\tind$ and $\Adot(\redb,\redC)$ denotes the right-hand side of \eqref{eq:redb}; thereby $\lVert\Pi_{\TM{\rsR_{\tind}}}\hvf(\rsR_\tind)\rVert_F^2=\lVert\Adot(\redb_\tind,\redC_\tind)\redC_\tind\rVert_F^2+\lVert\redb_\tind^\top \hvf(\rsR_\tind)\rVert_F^2$.
Since the quantities involved are already available from the solution of the reduced dynamics \eqref{eq:ROM}, the angle \eqref{eq:theta_errind} can be computed in $O(\Nfh\np\nr)$ operations. Although this might still be expensive when both $\np$ and $\Nfh$ are large, one may evaluate the angle based on a subset of $\npAs$ parameters rather than on all $\np$ parameters, with a computational complexity of $O(\Nfh\npAs\nr)$. We will discuss this option in \Cref{sec:pAstar_selection}. 
Moreover, the angle \eqref{eq:theta_errind} provides an indicator related to the error between the approximate and full order solution, see \Cref{prop:errind}. However, such error indicator shows some limitations at least in some numerical tests (see \Cref{sec:numexp}). Indeed, using an angle to determine when the approximation space needs to be changed does not appear to be a robust measure of the approximation error, hence making it difficult to set a criterion/tolerance
to determine when the approximation has deteriorated.
This is further exacerbated by the fact that this approach does not take into account the time evolution of the approximate solution, but only its value at a given time instant.

To overcome the limitations of the available rank-adaptive algorithms, we propose an error indicator, related to the projection error of the Hamiltonian, but
 defined~as
\begin{equation}\label{eq:residual_errind}
    \indres_\tind:=\sum_{\ell=\overline{\tind}}^{\tind}(t^{\ell+1}-t^\ell)\widetilde{\indres}_{\ell}\,,
\end{equation}
where $\widetilde{\indres}_{\ell}:=\widetilde{\indres}(\rsR_{\ell})$, $t^{\overline{\tind}}$ is the time when the last update was performed, and
\begin{equation}\label{eq:projerr_hvf}
    \widetilde{\indres}(\rsR):=\lVert\hvf(\rsR)-\Pi_{\TM{\rsR}}\hvf(\rsR)\rVert_F,\qquad\forall\, \rsR\in\Mn.
\end{equation}
The quantity \eqref{eq:residual_errind} approximates the time integral of the projection error of the Hamiltonian vector field between $t^{\overline{\tind}}$ and $t^{\tind}$, thus, unlike $\indtheta_\tind$, taking into account the evolution of the approximation error. The complexity of computing $\widetilde{r}_{\ell}$ at a given time step is again $O(\Nfh\np\nr)$, so that the same considerations made for $\indtheta_\tind$ also apply here. 
Moreover, although \eqref{eq:residual_errind} does not provide a direct indication of the approximation error introduced by dimensionality reduction, the following result holds.
\begin{proposition}\label{prop:errind}
    Let $\overline{t}\geq t_0$ be fixed and let $E(t)$ denote the error between the (hyper-)reduced and full order solutions, i.e.,
    $E(t):=\fsR(t)-\rsR(t)$ for any $t>\overline{t}$. Let the error indicators $\indtheta$ and $\widetilde{\indres}$ be defined as in \eqref{eq:theta_errind} and \eqref{eq:projerr_hvf}, respectively. Then
    \begin{align*}
        \lVert E(t)\rVert_F
        &\leq\lVert E(\overline{t})\rVert_Fe^{L(t-\overline{t})}+\int_{\overline{t}}^t\widetilde{\indres}(\rsR(\tau))e^{L(t-\tau)}\,d\tau\\
        &=\lVert E(\overline{t})\rVert_Fe^{L(t-\overline{t})}+\int_{\overline{t}}^t\lVert\nabla\Hcal(\rsR(\tau))\rVert_F\sin(\indtheta(\rsR(\tau)))e^{L(t-\tau)}\,d\tau,
    \end{align*}
    where $L=\big(\sum_{k=1}^\np\ell^2_k\big)^{1/2}$ and $\ell_k$ is the Lipschitz continuity constant of the Hamiltonian vector field $\hvfk$ associated with the $k$th parameter.
\end{proposition}
\begin{proof}
    The evolution of the error can be written as
    \begin{equation*}
        \Dot{E}=\Dot{\fsR}-\Dot{\rsR}=\hvf(\fsR(t))-\hvf(\rsR(t))+\hvf(\rsR(t))-\Pi_{\TM{\rsR}}\hvf(\rsR(t)).
    \end{equation*}
    Since $d_t\norm{E(t)}_F\leq\lVert\dot{E}(t)\rVert_F$, it holds $d_t\norm{E(t)}_F  \leq L\lVert E(t)\rVert_F+\widetilde{\indres}(\rsR(t))$
    and the result follows from Gr\"onwall's lemma \cite{Gro19}. Moreover, it can be easily shown that $\widetilde{\indres}(\rsR)=\norm{\nabla\Hcal(\rsR)}_F\sin{(\theta(\rsR))}$, because $\Pi_{\TM{\rsR}}$ is an orthogonal projection with respect to the Frobenius norm.
\end{proof}
This result shows that both $\indtheta_\tind$ and $\indres_\tind$ provide approximations to an upper bound of the error.
Therefore, we suggest the following procedure for rank adaptation. Assume that the last rank update has been performed at time $t^{\overline{\tind}}$ and $\overline{\inde}:=\inde_{\overline{\tind}}$ is the value of the error indicator at the last update, where $\inde$ denotes either $\indtheta$ or $\indres$.
As suggested in \cite{HPR22}, the rank update is performed when
\begin{equation}\label{eq:criterion}
\inde_\tind\geq \overline{\inde}K(\ladapt),\quad \mbox{where }\quad K(\ladapt):=\radapt\cadapt^{\ladapt},\quad \radapt,\cadapt\in\Rbb_+,\, \ladapt\in\mathbb{N}.
\end{equation}
In case this criterion is satisfied, the rank of the reduced solution is increased: the reduced basis is augmented with the space directions in which the projection error $E_P(\rsR_{\tind}):=\hvf(\rsR_\tind)-\Pi_{\TM{\rsR}}\hvf(\rsR_\tind)$ is currently worst approximated. More in details, we compute the singular value decomposition of $E_P(\rsR_{\tind})$ and add to the reduced basis the $\nrnew-\nr$ dominant singular vectors together with their symplectic duals.
The new vector of coefficients is defined as the projection of the old coefficients onto the new reduced basis: $\redC_\tind^{\new}=(\redb_\tind^{\new})^\top\redb_\tind\redC_\tind$. The number of new basis vectors to be added to the old reduced basis can either be fixed (e.g., by adding one pair $(v,J_{\Nf}^{\top}v)$ of basis vectors at each adaptation) or determined based on a suitable tolerance.

It is important to note that the matrix $S$ is ill-conditioned after the rank update, so that numerical inaccuracies may arise when solving the evolution equation for the reduced basis. For this reason, a suitable regularization needs to be performed before inverting $S$, as explained in \cite[Section 4]{HPR22}.
We point out that decreasing the size of the reduced basis space is a simpler task since one can look at the spectrum of the reduced state and remove the modes associated with the singular values smaller than a certain value.

\subsection{Adaptation of the dimension of the EIM space}
To initialize the EIM pair $(\deimb_0,\deimii_0)$, we collect instances of the reduced Jacobian at the initial time in the matrix $\deimsn_0:=\JGdec(\redb_0\redC_0)\redb_0\in\Rbb^{\nd\times 2\Nrz\np}$, assuming that $\deimsn_0\neq 0$.
If this is not the case, the reduced system \eqref{eq:ROM} is solved without hyper-reduction for $\deimf_0$ time steps, and $t^{\deimf_0}$ is then taken as initial time for the adaptive EIM strategy.
The initial EIM basis $\deimb_0\in\Rbb^{\nd\times\nmz}$ is computed via truncated SVD of $\deimsn_0$ to achieve a certain accuracy, and the interpolation points $\deimii_0$ are obtained from $\deimb_0$. 
The EIM basis is then adapted every $\deimf\geq1$ time steps as described in \Cref{sec:adaptEIM}.
Although the EIM basis is in principle adapted at each time, we propose to update its dimension only when the rank of the reduced solution is updated.
The rationale is that hyper-reduction is performed on the nonlinear Jacobian mapped to the reduced space: as the reduced space gets larger, it has been observed that the rank of the nonlinear operator increases as well \cite[Figure 4]{PV22}.
More in details, whenever the rank of the approximate state changes,
the EIM basis is re-computed via truncated SVD, up to a given tolerance $\tau_{\nm}$, of the matrix of the reduced Jacobian evaluated at the reduced solution $\redb^{\new}\redC^{\new}$ after the rank update.
In principle one can use subsampling techniques and perform the update based on the matrix $\redCs$ collecting a subset of the columns of $\redC$; the construction of $\redCs$ will be discussed in \Cref{sec:pAstar_selection}.
The rank-adaptive procedure is summarized in \Cref{alg:rank-update}.

\begin{algorithm}[H]
    \caption{Rank update at fixed time $t=t^{\tind}$}\label{alg:rank-update}
    \begin{algorithmic}[1]
        \Procedure{$(\redb^{\new},\redC^{\new},\deimb,\deimii)$=Rank\_update}{$\redb,\redC,\redC^*,\tau_{\nm}$}
        \State Form $\redb^{\new}$ by increasing $\redb$ with the first $\nrnew-\nr$ singular vectors of $E_P(\redb\redC^*)$
        \State $\redC^{\new}\gets(\redb^{\new})^\top\redb\redC$
        \State Compute $\deimb$ via truncated SVD of $\deimsn(\redb^{\new},(\redC^{\new})^*)$ \eqref{eq:Ni} with tolerance $\tau_{\nm}$
        \State $\deimii\gets \textsc{QDEIM}(\deimb)$
        \EndProcedure
    \end{algorithmic}
\end{algorithm}


\section{Parameters sampling for adaptivity}
\label{sec:pAstar_selection}

Even if the size of the reduced basis space satisfies $\Nrh(t)\ll\Nfh$ for all $t\in\Tcal$, the evolution of the reduced basis and the computation of the error indicator require evaluating the full order Hamiltonian vector field $\hvf$ at the reduced solution for all test parameters. The computational complexity of these operations scales with the product $\Nfh\np$. To mitigate this computational burden, we propose to select a subset of $\npAs<\np$ sample parameters ${\vprm}^*=\{\prm_{i_1},\dots,\prm_{i_{\npAs}}\}$
at each time $t^{\tind}$ and to consider the submatrix $\redCs_{\tind}$ containing the columns of the approximate coefficients $\redC_{\tind}$ of indices $\{i_1,\ldots,i_{\npAs}\}$ in the rank update \Cref{alg:rank-update}, and in the evolution of the reduced basis \eqref{eq:redb}.
In this way, all operations with arithmetic complexity $O(\Nfh\np)$ have complexity $O(\Nfh\npAs)$.

To select the sample parameters, we propose to compute the QR factorization of $\redC_\tind$ with column pivoting. The idea is that the reduced space is evolved by only considering the parameters that are ``most relevant" to the dynamics. Typically, we would like to have $\npAs$ much smaller than the number of test parameters $\np$, but we require that $\npAs\geq\nr$,
with $\Nr$ the dimension of the current reduced basis space.
This is a necessary condition 
to satisfy the rank condition $\rank(\redCs_\tind{\redCs_\tind}^\top+J_{\Nr}^\top \redCs_\tind{\redCs_\tind}^\top J_{\Nr})=\Nr$.


With the proposed algorithm, the hyper-reduced model
in each $\Tcalt$ reads
\begin{subequations}\label{eq:hROM}
\begin{empheq}[left = \empheqlbrace\,]{align}
& \dot{\redb}  = (I_{\Nf}-\redb\redb^\top)\big(J_{\Nf}\nabla\Hcal(\redb\redCs)(\redCs)^\top{+}\nabla\Hcal(\redb\redCs)(\redCs)^\top J_{\Nr}\big)M^{-1}(\redCs), 
\label{eq:redbs}\\ 
& \dot{\redC}  = J_{\Nr}\nabla\hrham_\tind(\redC), 
\label{eq:redCs}
\end{empheq}
\end{subequations}
where $\hrham_\tind$ is the hyper-reduced Hamiltonian (see \Cref{sec:adap_GPEIM}) with the subscript $\tind$ referring to the EIM approximation in the current temporal sub-interval $\Tcalt$.
Note that the approximate reduced model \eqref{eq:hROM} retains the geometric structure of the full order model.

\section{Temporal discretization of the hyper-reduced model}\label{sec:temp_discr_hypreduced}

For the numerical time integration of the hyper-reduced model \eqref{eq:hROM},
we adopt the temporal integrator proposed in \cite{P19}.
The idea is to combine a symplectic temporal integrator for the evolution \eqref{eq:redCs} of the expansion coefficient with a time discretization of the basis evolution \eqref{eq:redbs} able to preserve the ortho-symplectic constraint.

In each temporal sub-interval $\Tcalt=(t^{\tind},t^{\tind+1}]$, given the approximate reduced basis $\redb_{\tind}$,
the method constructs a local retraction $\retr_{\redb_{\tind}}$ so that $\redb(t)=\retr_{\redb_{\tind}}(\Vretr(t))$ for some $\Vretr$ in the tangent space at $\redb_{\tind}$ that is evolved in time.
Given $\redC_{\tind}$, the problem reads
\begin{equation}\label{eq:hROMtang}
    \left\{\begin{aligned}
       & \Dot{\Vretr}(t)=f_\tind(\Vretr(t),\redCs(t)), \\
       & \Dot{\redC}(t)=g^{\hr}_\tind(\Vretr(t),\redC(t),\deimb_\tind,\deimii_\tind),
    \end{aligned}\right.
\end{equation}
where $\Vretr(t^\tind)=0$, $f_\tind$ is defined as
$f_\tind=\big(d\retr_{{\redb_\tind}_|{_{\Vretr(t)}}}\big)^{-1}\Adot(\retr_{\redb_\tind}(\Vretr(t)),\redC(t))$ via the inverse tangent map of the retraction and 
\begin{equation*}
    g^{\hr}_\tind(\Vretr(t),\redC(t),\deimb_\tind,\deimii_\tind):=J_{\Nr}\nabla\hrham_\tind\big(\retr_{\redb_\tind}(\Vretr(t))\redC(t)\big).
\end{equation*}
Once $\Vretr_{\tind+1}$ has been computed, the reduced basis is obtained as $\redb_{\tind+1}=\retr_{\redb_\tind}(\Vretr_{\tind+1})$.

Since $\deimb_\tind$ and $\deimii_\tind$ are fixed, system \eqref{eq:hROMtang} can be solved using the 2-stage second order partitioned RK scheme introduced in \cite[Lemma A.1]{HPR22}, where
the equation for $V$ is solved using the explicit midpoint scheme and the equation for $\redC$ using the implicit midpoint rule. The application of the method to the hyper-reduced model \eqref{eq:hROM} is presented in \Cref{alg:partRKhrom}.
The initial reduced basis $\redb_0$ is obtained via complex SVD of the initial condition $\fsR_0$, while the initial coefficient matrix is given by the projection of the initial condition onto the reduced space, that is, $\redC_0:=\fsR_0-\redb_0\redb_0^\top\fsR_0$.
We refer to \cite{P19,HPR22} for further details on this temporal integrator. 

\begin{algorithm}[H]
    \caption{Partitioned RK2 scheme for the hyper-reduced model \eqref{eq:hROM}}\label{alg:partRKhrom}
    \begin{algorithmic}[1]
        \Procedure{$(\redb_{\tind+1},\redC_{\tind+1})$=\textsc{pRK2\_HROM}}{$\redb_\tind,\redC_\tind,\redC^*_\tind,\deimb_\tind,\deimii_\tind$}
        \State $\hki^*=\Adot(\redb_\tind,\redC^*_\tind)$        
        \State $\kii=g^{\hr}_\tind\left(\frac{\dt}{2}\hki^*,\redC_\tind+\frac{\dt}{2}\kii,\deimb_\tind,\deimii_\tind\right)$
        \State Restrict $\kii$ to the sample parameters: $\kii^*:=g^{\hr}_\tind\left(\frac{\dt}{2}\hki^*,\redC^*_\tind+\frac{\dt}{2}\kii^*,\deimb_\tind,\deimii_\tind\right)$
        \State $\hkii=f_\tind\left(\frac{\dt}{2}\hki^*,\redC^*_\tind+\frac{\dt}{2}k^*_2\right)$
        \State $\Vretr_{\tind+1}\gets\dt_{\tind}\,\hkii$ and $\redb_{\tind+1}\gets\retr_{\redb_\tind}(\Vretr_{\tind+1})$
        \State $\redC_{\tind+1}\gets\redC_\tind+\dt_{\tind}\,\kii$
        \EndProcedure
    \end{algorithmic}
\end{algorithm}


\section{Summary of the algorithm}\label{sec:summary}

The main steps of the proposed structure-preserving adaptive hyper-reduction method for the solution of the parametric system \eqref{eq:FOM} are reported in \Cref{alg:final-adaptive-alg}, whose computational complexity is detailed below. We report the highest possible cost for operations such as the solution of a linear system, matrix-matrix multiplication, etc. These operations can be clearly accelerated in situations where the problem has particular structures and sparsity patterns.


\begin{algorithm}
    \caption{Structure-preserving adaptive hyper-reduction}\label{alg:final-adaptive-alg}
    \begin{algorithmic}[1]
        \Procedure{SPA-hyper}{$\fsR_0$, $K$, $\nt$, $\tau_{\nmz}$}
        \State Compute $\redb_0$ via complex truncated SVD of $\fsR_0$
        \State $\redC_0\gets\redb_0^\top\fsR_0$
        \State $\ladapt\gets0$
        \State Compute $\deimb_0$ via truncated SVD of $\JGdec(\redb_0\redC_0)\redb_0$ with tolerance $\tau_{\nm}$
        \State $\deimii_0\gets\textsc{QDEIM}(\deimb_0)$
        \For {$\tind=0,\dots,\nt-1$}
        \State Form $\vprm^*_{\tind}$ by selecting $\npAs$ parameters via pivoted QR of $\redC_{\tind}$
        \State Evaluate the error indicator $\inde^*_{\tind}$ from $\rsR_{\tind}^*=\redb_{\tind}\redC_{\tind}^*$
        \If{$\tind=0$}
        \State $\overline{\inde}\gets\inde^*_{\tind}$
        \ElsIf{$\tind>0$ and $\inde^*_\tind>\overline{\inde}K(\ladapt)$}
        \State $(\redb_{\tind},\redC_{\tind},\deimb_\tind,\deimii_\tind)=\textsc{Rank\_update}(\redb_\tind,\redC_\tind,\redC_{\tind}^*,\tau_{\nm})$ as in \Cref{alg:rank-update}
        \State $\overline{\inde}\gets \inde^*_\tind$, \;$\ladapt\gets\ladapt+1$
        \State  Form $\redC_{\tind}^*$ with the columns of $\redC_{\tind}$ associated with the parameters $\vprm^*_{\tind}$
        \EndIf
        \State $(\redb_{\tind+1},\redC_{\tind+1})=\textsc{pRK2\_HROM}(\redb_\tind,\redC_\tind,\redC^*_\tind,\deimb_\tind,\deimii_\tind)$ as in \Cref{alg:partRKhrom}
        \State $(\deimb_{\tind+1},\deimii_{\tind+1})=\textsc{EIM\_upd}(\redb_{\tind+1},\redC_{\tind+1},\deimb_\tind,\deimii_\tind,\tau_{\nms})$ as in \Cref{alg:DEIM_update}
        \EndFor
        \EndProcedure
    \end{algorithmic}
\end{algorithm}

The computation of the initial data (lines 2-6) requires $O(\Nfh\np\Nrh_0)+O(\np\nd\nmz\Nrh_0)+O(\nd\nmz^2)$ operations.
Then, at each time step, \Cref{alg:final-adaptive-alg} incurs the following costs:
\begin{itemize}
    \item $O(\np\npAs\nr)$ for the pivoted QR of $\redC_{\tind}$, and
    $O(\Nfh\npAs\nr)$ for the evaluation of the error indicator $\inde^*_{\tind}$.
    Note that, in practice, we use the information on the parameters computed with \Cref{alg:DEIM_update} at line 18, whenever available.
    \item $O(\Nfh\nrp{2})+O(\Nfh\npAs\nr)+O(\np\nrp{3}\nINL)+O(\np\nr\nnzJ\nm\nINL)$ for \Cref{alg:partRKhrom}, where $\nINL$ is the number of iterations of the nonlinear solver. More in details, the evaluation of the velocity $\Adot$ at $\npAs$ sample parameters has leading complexity $O(\Nfh\nrp{2})+O(\Nfh\npAs\nr)$ \cite[Section 6]{HPR22}.
    Next, $\nabla\Hcal_\tind^{\hr}(A,Z)$ can be computed with $O(\nm^3)+O(\Nfh\nm)+O(\nm\nnzJ\nr)$ operations. Since $\nnzJ$ is also the number of non-zero entries in each row of $\nabla^2\Hcal_\tind^{\hr}$, the hyper-reduced Hessian can be evaluated with $O(\nm\nnzJ\nr)$ operations.
    %
    For the computation of $g^{\hr}_\tind$ a nonlinear system has to be solved for each of the $\np$ columns of $\kii$. If Newton's method is employed, each iteration consists of evaluating $\nabla^2\Hcal_\tind^{\hr}$ and solving a $\Nr\times\Nr$ system for all test parameters, with complexity $O(\np\nr\nnzJ\nm\nINL)+O(\np\nrp{3}\nINL)$. Finally, with the algorithm \cite[Section 5.3.1]{P19}, the computation of the retraction $\retr_{\redb_\tind}$, its inverse tangent map and the assembly of the operator $f_{\tind}$ have arithmetic complexity $O(\Nfh\nrp{2})$.
    \item $O(\Nfh\npUs\nr)+O(\nd\npUs\nr\nrc)+O(\nd\nm^2)+O(\np\npUs\nr)+O(\np\nr\nm\nms)$ for \Cref{alg:DEIM_update} as derived in \Cref{sec:samples}.
\end{itemize}
In addition, at each rank update, \Cref{alg:rank-update} is applied with arithmetic complexity $O(\Nfh(\npAs)^2)+O(\nd\nrnew\npAs\nnzJ)+O(\nd\nrnew\npAs\nm)+O(\nd\nm^2)$,
resulting from the SVD of the projection error $E_P(\fsR^*)\in\Rbb^{\Nf\times\npAs}$, the assembly of $\deimsn(\redb^{\new},(\redC^{\new})^*)$ and the computation of its truncated SVD with tolerance $\tau_\nm$, and the computation of the new DEIM interpolation indices.
Note that the matrix-matrix multiplications involved in the definition of $\redC^{\new}$ can be avoided
by augmenting $\redC$ with $\nrnew-\nr$ rows of zeros.

A reduction in the computational runtimes of the algorithm compared to the solution of the reduced dynamics \eqref{eq:ROM} is thus achieved by decoupling the operations that depend on the dimension of the full order model from those that depend on the number of parameters.


\section{Numerical experiments}\label{sec:numexp}
In the two-dimensional, rectangular domain $\Omega=[-L_\hc,L_\hc]\times[-L_\vc,L_\vc]\subset\Rbb^2$, we consider the nonlinear Schr\"odinger equation
\begin{equation}\label{eq:schrodinger}
        \imath\partial_t u+\Delta u+\prme\lvert u\rvert^2u=0 \qquad \text{ in } \Omega\times\mathcal{T}, 
\end{equation}
with periodic boundary conditions and initial condition $u(0,\hc,\vc;\prmaic,\prmbic)=u^0(\hc,\vc;\prmaic,\prmbic)$.
The solution depends on the parameter vector $\prm=(\prmaic,\prmbic,\prme)\in\prms\subset\Rbb^3$, where $\prme$ enters the equation, while $\prmaic$ and $\prmbic$ enter the initial condition. The Hamiltonian formulation of this problem is obtained by writing $u(t,\hc,\vc;\prm)=q(t,\hc,\vc;\prm)+\imath p(t,\hc,\vc;\prm)$ and deriving evolution equations for the real and imaginary parts, $q$ and $p$ respectively:
\begin{equation}\label{eq:NLSfom}
    \left\{\begin{aligned}
        &\partial_t q=-\Delta p-\prme(q^2+p^2)p & \text{ in } \Omega\times\Tcal, \\
        &\partial_t p=\Delta q+\prme(q^2+p^2)q & \text{ in } \Omega\times\Tcal. 
    \end{aligned}\right.
\end{equation}
This is a Hamiltonian system with Hamiltonian given by
\begin{equation*}
    \frac{1}{2}\int_\Omega\left(\lvert\nabla q\rvert^2+\lvert\nabla p\rvert^2-\frac{\prme}{2}(q^2+p^2)^2\right)\,d\hc\,d\vc.
\end{equation*}
Problem \eqref{eq:NLSfom} is discretized in space using a second-order finite difference scheme on a uniform grid with $N_\hc$ and $N_\vc$ intervals in the $\hc$ and $\vc$ direction, respectively. 
The dimension of the full order problem is then $\Nf$ with $\Nfh:=N_\hc N_\vc$. Defining the vectors $\fsq\in\Rbb^{\Nfh}$ and $\fsp\in\Rbb^{\Nfh}$ approximating the nodal values of $q$ and $p$, respectively, the full order system reads
\begin{equation*}
    \Dot{\fsy}=\begin{pmatrix}
        \dfsq\\
        \dfsp
    \end{pmatrix}
    =J_{\Nf}\begin{pmatrix}
        -\Dbb \fsq-\prme(\fsq^2+\fsp^2)\odot \fsq \\
        -\Dbb \fsp-\prme(\fsq^2+\fsp^2)\odot \fsp
    \end{pmatrix}=J_{\Nf}\nabla_\fsy\Hcal_{\prm}(\fsy),
\end{equation*}
where $\Dbb\in\Rbb^{\Nfh\times\Nfh}$ is the discrete Laplacian matrix, $\odot$ denotes the component-wise product of vectors, $\fsq^2:=\fsq\odot\fsq$ and $\Hcal_{\prm}$ is the discrete Hamiltonian defined as
\begin{equation}\label{eq:HamNLS}
    \Hcal_{\prm}(\fsy):=\frac{1}{2}\big(-\fsq^\top\Dbb \fsq-\fsp^\top\Dbb\fsp-\frac{\prme}{2}\sum_{i=1}^{\Nfh}((\fsq)_i^2+(\fsp)_i^2)^2\big).
\end{equation}
Following the procedure described in \Cref{sec:adap_GPEIM}, we decompose
the non-quadratic part of the Hamiltonian \eqref{eq:HamNLS} as $\Hcal_{\prm}(\fsy)=\cdec^\top\Gdec(\fsy)$ where $\cdec$ is the vector of $\Rbb^\Nfh$ whose entries are all $1$ and 
\begin{equation*}
    \Gdec(\fsy):=-\frac{\prme}{4}(\fsq^2+\fsp^2)^2\in\Rbb^\Nfh.
\end{equation*}
Note that we have $\nd=\Nfh$ in this case.

In the numerical experiments we set $L_\hc=L_\vc=2\pi$ and $N_\hc=N_\vc=100$, so that the full system has dimension $\Nf=2N_\hc N_\vc=20000$. The final time is $T=3$ and we consider $\nt=12000$ uniform temporal intervals so that $\Delta t=2.5\times10^{-4}$. Note that the choice of such a relatively small time step is not dictated by stability issues but only to facilitate the study of the approximation error associated with the reduction and hyper-reduction without pollution from the temporal integration error.
The full order system is solved using the implicit midpoint rule, while the reduced and hyper-reduced systems are solved using the partitioned RK scheme described in \Cref{sec:temp_discr_hypreduced}.
Since the proposed time integrators are (semi-)implicit, a nonlinear system has to be solved for each test parameter at each time step: the Newton method with stopping tolerance $\tau=10^{-10}$ is chosen as nonlinear solver. This operation is parallelized in the numerical solution of the full and reduced system
using $16$ cores while no parallelization is performed on the hyper-reduced system. In fact, in the latter case the runtime to solve the resulting nonlinear system is so small that no computational gain is achieved by parallelizing the computation because of overhead costs.
For the sake of a fair comparison, the computational time per core per test parameter is taken as a unit of measure, in the sense that the computational time required by each operation is multiplied by the number of cores on which the operation is executed.
We assess the accuracy of the different algorithms in terms of relative errors in the Frobenius norm with respect to the full order solution, namely
\begin{equation*}
    \Ecal(t):=\frac{\norm{\fsR(t)-\rsR(t)}_F}{\norm{\fsR(t)}_F}
\end{equation*}
where $\rsR$ can be the numerical solution of the reduced system \eqref{eq:ROM} or of the hyper-reduced system \eqref{eq:hROM}.
Finally, we measure the error in the conservation of the Hamiltonian with respect to the initial condition via the average relative error
\begin{equation}\label{eq:err_hamcons}
    \Ecal_\Hcal(t):=\frac{1}{p}\sum_{i=1}^p\left\lvert\frac{\Hcal_{\eta_i}(W_i(t))-\Hcal_{\eta_i}(W_i(0))}{\Hcal_{\eta_i}(W_i(0))}\right\rvert.
\end{equation}
We shall consider two different test cases for problem \eqref{eq:schrodinger}. First, in \Cref{sec:exp_localized}, we choose the initial condition so that the solution is localized in space, and its numerical rank exhibits a moderate growth over time. Next, in \Cref{sec:exp_nonlocalized} we conduct a more challenging experiment on a non-localized solution, whose numerical rank increases steadily during the simulation. Our goal is to compare the performance of hyper-reduction in the two cases, and to show that our rank-adaptive algorithm is able to correctly follow the evolution of the numerical rank.


\subsection{Test 1: Localized solution}\label{sec:exp_localized}
In the first scenario, we consider as initial condition of problem \eqref{eq:schrodinger}
a hump placed at the center of the computational domain 
\begin{equation}\label{eq:ICtest1}
    u^0(\hc,\vc;\prmaic,\prmbic)=\frac{\sqrt{2}}{\cosh(\prmaic \hc)\cosh(\prmbic \vc)}\exp{\left(\imath\displaystyle\frac{\hc}{2}\right)}\exp{\left(\imath\displaystyle\frac{\vc}{2}\right)}.
\end{equation}
The parameter set for this problem is $\prms=[0.8,2]\times[0.8,2]\times[-1.5,-0.5]$; the set of test parameters is obtained by discretizing $\prms$ with $5$ uniformly spaced parameters in each direction, so that $p=5^3=125$. \Cref{fig:localized_fullsol_rankfullsol} (left) shows the numerical solution corresponding to one test parameter.
The time required to solve the full order system is $5241\,s$. 
\begin{figure}[h]
\begin{minipage}{0.6\linewidth}
    \centering
    \includegraphics[width=\textwidth]{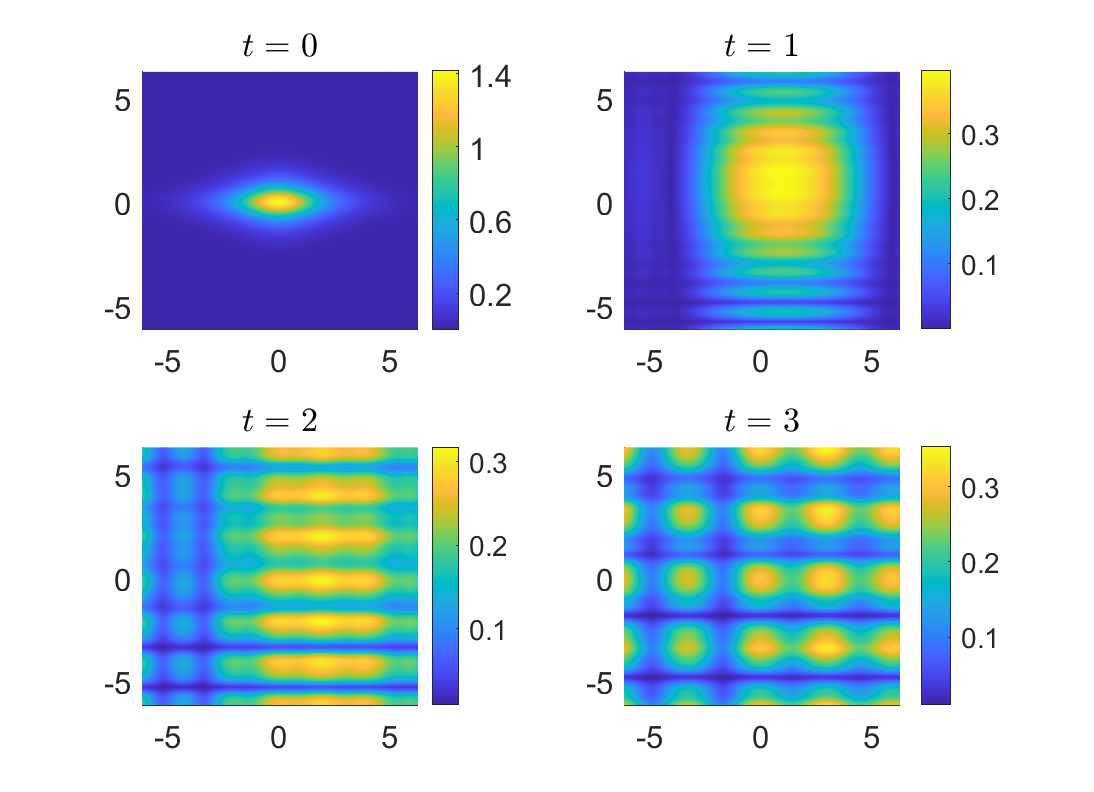}
\end{minipage}
\begin{minipage}{0.39\linewidth}
   \centering
\begin{tikzpicture}
    \begin{groupplot}[
      group style={group size=1 by 1,
                  horizontal sep=1.4cm},
      width=\textwidth, height=4.5cm
    ]
    \nextgroupplot[ylabel={$\epsilon$-rank},
                  xlabel={$t$},
                  xlabel style = {yshift=.2cm},
                  ylabel style = {yshift=-.5cm},
                  axis line style = thick,
                  grid=both,
                  minor tick num=0,
                  grid style = {gray,opacity=0.2},
                  xmin=0, xmax=3,
                  ymin=0, ymax=140,
                  xlabel style={font=\footnotesize},
                  ylabel style={font=\footnotesize},
                  x tick label style={font=\footnotesize},
                  y tick label style={font=\footnotesize},
                  legend style={font=\footnotesize},
                  legend cell align={left},
                  legend columns = 2,
                  legend style={at={(0.5,1.5)},anchor=north}]
        \addplot+[color=red,mark=none,line width=1.5pt] table[x=t,y=tau1e-1] {localized_rankfullsol_T.txt};
        \addplot+[color=green,mark=none,line width=1.5pt] table[x=t,y=tau1e-3] {localized_rankfullsol_T.txt};
        \addplot+[color=blue,mark=none,line width=1.5pt] table[x=t,y=tau1e-5] {localized_rankfullsol_T.txt};
        \addplot+[color=violet,mark=none,line width=1.5pt] table[x=t,y=tau1e-7] {localized_rankfullsol_T.txt};
        \addplot+[color=orange,mark=none,line width=1.5pt] table[x=t,y=tau1e-9] {localized_rankfullsol_T.txt};
        \legend{{$\epsilon=1e{-}01$}\;,{$\epsilon=1e{-}03$}\;,{$\epsilon=1e{-}05$}\;,{$\epsilon=1e{-}07$}\;,{$\epsilon=1e{-}09$}};
    \end{groupplot}
\end{tikzpicture}
  \end{minipage}
\caption{\footnotesize Test 1, localized solution. Left: modulus of the full order solution at four different time instants for $\prm=(0.8,2,-1.5)$. Right: $\epsilon$-rank of the full solution $\fsR(t)\in\Rbb^{20000\times 125}$.}\label{fig:localized_fullsol_rankfullsol}
\end{figure}
%
This test case is characterized by a growth of the numerical rank in the first stages of the simulation, as shown in \Cref{fig:localized_fullsol_rankfullsol} (right). We expect the error in the non-adaptive reduced and hyper-reduced systems to possess a similar behavior, while the rank-adaptive strategy should correctly capture this evolution and keep the error constant over time.

\subsubsection{Non-rank-adaptive algorithm}
In a first test, we assess the capability of the hyper-reduction strategy to reduce the computational cost of the reduced system while achieving a comparable error. To this end, we consider the non-rank-adaptive case by fixing the dimension of the reduced space, i.e., $2\nr=2\Nrz=24$ for $\tind=1,\ldots,\nt$.
The error of the reduced system is shown in \Cref{fig:localized_nonra_n12_m80}, left (dashed line):
the relative error is below $10^{-3}$ at the initial time, then increases rapidly at the beginning of the simulation, while the growth becomes less evident starting from about $t=0.5$. This reflects the behavior of the numerical rank of the full solution, see \Cref{fig:localized_fullsol_rankfullsol}. The hyper-reduced system is solved using the adaptive strategy described in \Cref{sec:adaptEIM}. In particular, the EIM adaption is performed at each time step with $\tau_{\nms}=10^{-4}$. We determine the dimension of the EIM space by performing POD on the nonlinear term at the initial time with tolerance $\tau_\nm=10^{-8}$, which results in $\nm_\tind=\nm_0=80$. The relative error of the hyper-reduced system is shown in \Cref{fig:localized_nonra_n12_m80} (left) for different choices of $\npAs$ and $\npUs$, whose evolution is reported in
\Cref{fig:localized_nonra_n12_m80} (right).

\begin{figure}[ht]
\centering
\begin{tikzpicture}
    \begin{groupplot}[
      group style={group size=2 by 1,
                  horizontal sep=1.4cm},
      width=6.5cm, height=5cm
    ]
    \nextgroupplot[ylabel={$\mathcal{E}(t)$},
                  xlabel={$t$},                  
                  xlabel style = {yshift=.2cm},
                  ylabel style = {yshift=-.2cm},
                  axis line style = thick,
                  grid=both,
                  minor tick num=0,
                  ytick={5e-04,1e-03,5e-03},
                  grid style = {gray,opacity=0.2},
                  xmin=0, xmax=3,
                  ymin=5e-04, ymax=5e-03,
                  ymode=log,
                  xlabel style={font=\footnotesize},
                  ylabel style={font=\footnotesize},
                  x tick label style={font=\footnotesize},
                  y tick label style={font=\footnotesize},
                  legend style={font=\footnotesize},
                  legend cell align={left},
                  legend columns = 1,
                  legend style={at={(0.55,0.65)},anchor=north}]
        \addplot+[color=black,dashed,mark=none,line width=1.5pt] table[x=t,y=red] {localized_nonra_errors_n12.txt};
        \addplot+[color=ForestGreen,mark=none,line width=1.5pt] table[x=t,y=hred_pApU] {localized_nonra_errors_n12.txt};
        \addplot+[color=Cyan,mark=none,line width=1.5pt] table[x=t,y=hred_pAspU] {localized_nonra_errors_n12.txt};
        \addplot+[color=YellowOrange,mark=none,line width=1.5pt] table[x=t,y=hred_pApUs] {localized_nonra_errors_n12.txt};
        \addplot+[color=red,mark=none,line width=1.5pt] table[x=t,y=hred_pAspUs] {localized_nonra_errors_n12.txt};
        \legend{{ROM},{hROM, $\npAs=\np,\npUs=\np$},{hROM, $\npAs<\np,\npUs=\np$},{hROM, $\npAs=\np,\npUs<\np$},{hROM, $\npAs<\np,\npUs<\np$}};
    \nextgroupplot[ylabel={$\npUs,\npAs$},
                  xlabel={$t$},
                  xlabel style = {yshift=.2cm},
                  ylabel style = {yshift=-.4cm},
                  axis line style = thick,
                  grid=both,
                  minor tick num=0,
                  grid style = {gray,opacity=0.2},
                  xmin=0, xmax=3,
                  ymin=0, ymax=25,
                  xlabel style={font=\footnotesize},
                  ylabel style={font=\footnotesize},
                  x tick label style={font=\footnotesize},
                  y tick label style={font=\footnotesize}]   
        \addplot+[color=Cyan,dashed,mark=none,line width=1.5pt] table[x=tpAs_pAspU,y=pAs_pAspU]{localized_nonra_pAstar_pUstar_n12.txt};
        \addplot+[color=YellowOrange,mark=none,line width=1.5pt] table[x=tpUs_pApUs,y=pUs_pApUs] {localized_nonra_pAstar_pUstar_n12.txt};
        \addplot+[color=red,mark=none,line width=1.5pt] table[x=tpUs_pAspUs,y=pUs_pAspUs] {localized_nonra_pAstar_pUstar_n12.txt};
        \addplot+[color=red,dashed,dash phase=3pt,mark=none,line width=1.5pt] table[x=tpAs_pAspUs,y=pAs_pAspUs] {localized_nonra_pAstar_pUstar_n12.txt};
    \end{groupplot}
\end{tikzpicture}
\caption{\footnotesize Test 1, localized solution. Dimension of the dynamical reduced basis space $\Nr=2\Nrz=24$. Left: evolution of the relative errors of the reduced and hyper-reduced systems for $\nm_\tind=\nm_0=80$ and with different choices of the sample parameters. Right: evolution of $\npUs$ (solid lines) and of $\npAs$ (dashed lines).}\label{fig:localized_nonra_n12_m80}
\end{figure}

\Cref{fig:localized_times_err_ms} (left) reports the computational times and the relative errors at the final time. 
In all cases considered, the computational times are reduced by at least a factor $13$ compared to the reduced system with an almost negligible effect on the relative error. Moreover,
the hyper-reduced system with $\npAs<\np$ and $\npUs<\np$ is solved in about half of the time required to perform the same task with $\npAs=\npUs=\np$. 
In this experiment, the use of a small number $\npUs$ of sample parameters for EIM adaptivity leads to a considerable computational gain. In particular, the quantity $\dsc$ defined in \Cref{sec:adaptEIM} can be evaluated cheaply
since only about $\nms=1000\ll d$ sample indices need to be selected when the tolerance is set to $\tau_{\nms}=10^{-4}$ (see \Cref{fig:localized_times_err_ms}, right).
This is related to the decay of local coherence of the EIM space: as noted in \cite[Section 3.3]{Peher20}, when a problem is characterized by moving coherent structures that are local in space, it is typically sufficient to adapt the EIM basis at a few sample indices only.


\begin{figure}[h]
\begin{minipage}{0.5\linewidth}
{\footnotesize
    \begin{center}
    \begin{tabular}{|c|c|c|}\hline
        Model & Time [s]  & $\Ecal(T)$ \\ \hline
        ROM & 675.7 & 3.12e-03 \\ \hline
        hROM, $\npAs=\np,\npUs=\np$ & 50.5 & 3.28e-03 \\
        hROM, $\npAs<\np,\npUs=\np$ & 45.6 & 3.64e-03 \\
        hROM, $\npAs=\np,\npUs<\np$ & 34.0 & 3.69e-03 \\
        hROM, $\npAs<\np,\npUs<\np$ & 27.5 & 4.14e-03 \\ \hline
    \end{tabular}
    \end{center}
}
\end{minipage}\hfill
\begin{minipage}{0.45\linewidth}
   \centering
    \begin{tikzpicture}
    \begin{groupplot}[
      group style={group size=1 by 1,
                  horizontal sep=1.4cm},
      width=5.5cm, height=4.5cm
    ]
    \nextgroupplot[ylabel={$\nms$},
                  xlabel={$t$},
                  xlabel style = {yshift=.2cm},
                  ylabel style = {yshift=-.2cm},
                  axis line style = thick,
                  grid=both,
                  minor tick num=0,
                  grid style = {gray,opacity=0.2},
                  xmin=0, xmax=3,
                  ymin=0, ymax=10000,
                  xlabel style={font=\footnotesize},
                  ylabel style={font=\footnotesize},
                  x tick label style={font=\footnotesize},
                  y tick label style={font=\footnotesize},
                  legend style={font=\footnotesize},
                  legend cell align={left},
                  legend columns = 1,
                  legend style={at={(0.39,1)},anchor=north}]
        \addplot+[color=ForestGreen,mark=none,each nth point={10},line width=0.5pt] table[x=t,y=pApU] {localized_nonra_ms_n12.txt};
        \addplot+[color=Cyan,mark=none,each nth point={10},line width=0.5pt] table[x=t,y=pAspU] {localized_nonra_ms_n12.txt};
        \addplot+[color=YellowOrange,mark=none,each nth point={10},line width=0.5pt] table[x=t,y=pApUs] {localized_nonra_ms_n12.txt};
        \addplot+[color=red,mark=none,each nth point={10},line width=0.5pt] table[x=t,y=pAspUs] {localized_nonra_ms_n12.txt};
        \legend{{$\npAs=\np,\npUs=\np$},{$\npAs<\np,\npUs=\np$},{$\npAs=\np,\npUs<\np$},{$\npAs<\np,\npUs<\np$}};
    \end{groupplot}
    \end{tikzpicture}
  \end{minipage}
\caption{Test 1, localized solution. Non-rank-adaptive case, $\Nr=2\Nrz=24$, $\nm_\tind=\nm_0=80$. Left: computational times and relative errors at the final time. Right: evolution of the number $\nms$ of EIM sample indices selected according to \Cref{alg:DEIM_update} with $\tau_{\nms}=10^{-4}$.}\label{fig:localized_times_err_ms}
\end{figure}

\subsubsection{Rank-adaptive algorithm}
In a second test we numerically show how the rank-adaptive strategy of \Cref{sec:rank-adaptivity} can be used to overcome the growth of the approximation errors introduced by the reduction. The rank-adaptive algorithm uses the error indicator $\indres$ defined in \eqref{eq:residual_errind} and computed at every time step with $\radapt=\cadapt=1.1$ for the update criterion. Numerical tests (not reported here) show that the choice of these parameters has little impact on the quality of the approximation (see also \cite{HPR22}).
We consider the cases $2\Nrz=24$ and $2\Nrz=40$.
Whenever the dimension $\Nr$ of the reduced basis space is updated, the EIM dimension $\nm$ is modified as in \Cref{alg:rank-update}
with tolerance $\tau_\nm=10^{-4}$
and $\tau_\nm=10^{-6}$, respectively.
The initial value of $\nm$ is determined based on the same tolerance, which gives $\nm_0=38$ and $\nm_0=83$, respectively. Finally, we consider $\npAs<\np$ and $\npUs<\np$ sample parameters to solve the hyper-reduced system with rank adaptivity. \Cref{fig:localized_ra} (left) shows the evolution of the relative errors in the reduced and hyper-reduced system, and the comparison with the same quantities without rank adaptivity (indicated with ``NRA'' in the legend).
We also report the evolution of the values of the reduced dimension $\Nr$ (central plot) and of the EIM dimension $\nm_\tind$ (right plot). 

\begin{figure}[ht]
\centering
\begin{tikzpicture}
    \begin{groupplot}[
      group style={group size=3 by 2,
                  horizontal sep=1.4cm, vertical sep=.5cm},
      width=4.5cm, height=4.5cm,
    ]
    \nextgroupplot[ylabel={$\Ecal(t)$},
                  ylabel style = {yshift=-.2cm},
                  axis line style = thick,
                  grid=both,
                  minor tick num=0,
                  ytick={5e-04,1e-03,5e-03},
                  ymode=log,
                  grid style = {gray,opacity=0.2},
                  xmin=0, xmax=3,
                  ymin=5e-04, ymax=5e-03,
                  xlabel style={font=\footnotesize},
                  ylabel style={font=\footnotesize},
                  x tick label style={font=\footnotesize},
                  y tick label style={font=\footnotesize},
                  legend style={font=\footnotesize},
                  legend cell align={left},
                  legend columns = 2,
                  legend style={at={(1.95,1.4)},anchor=north}]
        \addplot+[color=black,mark=none,line width=1.5pt] table[x=t,y=red] {localized_ra_errors_n12.txt};
        \addplot+[color=red,mark=none,line width=1.5pt] table[x=t,y=hred] {localized_ra_errors_n12.txt};
        \addplot+[color=black,dashed,mark=none,line width=1.5pt] table[x=t,y=red] {localized_nonra_errors_n12.txt};
        \addplot+[color=red,dashed,mark=none,line width=1.5pt] table[x=t,y=hred_pAspUs] {localized_nonra_errors_n12.txt};
        \legend{{ROM-RA},{hROM-RA, $\npAs<\np$, $\npUs<\np$},{ROM-NRA},{hROM-NRA, $\npAs<\np$, $\npUs<\np$}};
        \nextgroupplot[ylabel={$\Nr$},
                  ylabel style = {yshift=-.5cm},
                  axis line style = thick,
                  grid=both,
                  minor tick num=0,
                  grid style = {gray,opacity=0.2},
                  xmin=0, xmax=3,
                  ymin=20, ymax=45,
                  xlabel style={font=\footnotesize},
                  ylabel style={font=\footnotesize},
                  x tick label style={font=\footnotesize},
                  y tick label style={font=\footnotesize}]
        \addplot+[color=black,mark=none,line width=1.5pt] table[x=tred,y=red2n] {localized_ra_2n_m_n12.txt};
        \addplot+[color=red,mark=none,line width=1.5pt] table[x=thred,y=hred2n] {localized_ra_2n_m_n12.txt};
        \addplot+[color=black,dashed,mark=none,line width=1.5pt] table[x=tna,y=2nna] {localized_ra_2n_m_n12.txt};
        \addplot+[color=red,dashed,dash phase=3pt,mark=none,line width=1.5pt] table[x=tna,y=2nna] {localized_ra_2n_m_n12.txt};
    \nextgroupplot[ylabel={$\nm_\tind$},
                  ylabel style = {yshift=-.4cm},
                  axis line style = thick,
                  grid=both,
                  minor tick num=0,
                  grid style = {gray,opacity=0.2},
                  xmin=0, xmax=3,
                  ymin=0, ymax=400,
                  xlabel style={font=\footnotesize},
                  ylabel style={font=\footnotesize},
                  x tick label style={font=\footnotesize},
                  y tick label style={font=\footnotesize}]
        \addplot+[color=red,mark=none,line width=1.5pt] table[x=thred,y=hredm] {localized_ra_2n_m_n12.txt};
        \addplot+[color=red,dashed,mark=none,line width=1.5pt] table[x=tna,y=mna] {localized_ra_2n_m_n12.txt};
    \nextgroupplot[ylabel={$\Ecal(t)$},
                  xlabel={$t$},
                  xlabel style = {yshift=.2cm},
                  ylabel style = {yshift=-.2cm},
                  axis line style = thick,
                  grid=both,
                  minor tick num=0,
                  minor tick style={draw=none},
                  ytick={1e-05,1e-04,1e-03,1e-02},
                  ymode=log,
                  grid style = {gray,opacity=0.2},
                  xmin=0, xmax=3,
                  ymin=1e-05, ymax=1e-02,
                  xlabel style={font=\footnotesize},
                  ylabel style={font=\footnotesize},
                  x tick label style={font=\footnotesize},
                  y tick label style={font=\footnotesize},
                  legend style={font=\footnotesize},
                  legend cell align={left},
                  legend columns = 2,
                  legend style={at={(1.95,1.4)},anchor=north}]
        \addplot+[color=black,mark=none,line width=1.5pt] table[x=t,y=rared] {localized_ra_errors_n20.txt};
        \addplot+[color=red,mark=none,line width=1.5pt] table[x=t,y=rahred] {localized_ra_errors_n20.txt};
        \addplot+[color=black,dashed,mark=none,line width=1.5pt] table[x=t,y=nrared] {localized_ra_errors_n20.txt};
        \addplot+[color=red,dashed,mark=none,line width=1.5pt] table[x=t,y=nrahred] {localized_ra_errors_n20.txt};        
    \nextgroupplot[ylabel={$\Nr$},
                  xlabel={$t$},
                  xlabel style = {yshift=.2cm},
                  ylabel style = {yshift=-.5cm},
                  axis line style = thick,
                  grid=both,
                  minor tick num=0,
                  grid style = {gray,opacity=0.2},
                  xmin=0, xmax=3,
                  ymin=35, ymax=60,
                  xlabel style={font=\footnotesize},
                  ylabel style={font=\footnotesize},
                  x tick label style={font=\footnotesize},
                  y tick label style={font=\footnotesize}]
        \addplot+[color=black,mark=none,line width=1.5pt] table[x=tred,y=red2n] {localized_ra_2n_m_n20.txt};
        \addplot+[color=red,mark=none,line width=1.5pt] table[x=thred,y=hred2n] {localized_ra_2n_m_n20.txt};
        \addplot+[color=black,dashed,mark=none,line width=1.5pt] table[x=tna,y=2nna] {localized_ra_2n_m_n20.txt};
        \addplot+[color=red,dashed,dash phase=3pt,mark=none,line width=1.5pt] table[x=tna,y=2nna] {localized_ra_2n_m_n20.txt};
    \nextgroupplot[ylabel={$\nm_\tind$},
                  xlabel={$t$},
                  xlabel style = {yshift=.2cm},
                  ylabel style = {yshift=-.4cm},
                  axis line style = thick,
                  grid=both,
                  minor tick num=0,
                  grid style = {gray,opacity=0.2},
                  xmin=0, xmax=3,
                  ymin=0, ymax=600,
                  xlabel style={font=\footnotesize},
                  ylabel style={font=\footnotesize},
                  x tick label style={font=\footnotesize},
                  y tick label style={font=\footnotesize}]
        \addplot+[color=red,mark=none,line width=1.5pt] table[x=thred,y=hredm] {localized_ra_2n_m_n20.txt};
        \addplot+[color=red,dashed,mark=none,line width=1.5pt] table[x=tna,y=mna] {localized_ra_2n_m_n20.txt};
    \end{groupplot}
\end{tikzpicture}
\caption{\footnotesize Test 1, localized solution. First row: $2\Nrz=24$, $\tau_\nm=\tau_{\nms}=10^{-4}$. Second row: $2\Nrz=40$, $\tau_\nm=\tau_{\nms}=10^{-6}$. Left: evolution of the relative errors of the reduced and hyper-reduced systems. Center: evolution of the dimension of the reduced space. Right: evolution of the dimension of the EIM space.}\label{fig:localized_ra}
\end{figure}

We observe that the value of $\Nr$ is adapted frequently in the early stages of the simulation, while the update frequency decreases sensibly as time evolves. In this sense, our adaptive strategy is able to correctly capture the behavior of the numerical rank of the full solution reported in \Cref{fig:localized_fullsol_rankfullsol}. Similarly, we observe a variation of the EIM dimension $\nm_\tind$, with its value at the final time being about $3\%$ and $6\%$ of the full dimension $\nd=\Nfh$ for $2\Nrz=24$ and $2\Nrz=40$, respectively. This is sufficient to yield a computational gain compared to the reduced system, as it can be inferred from the data reported in \Cref{tab:localized_ra}. The rank-adaptive algorithm for the hyper-reduced system is about $17$ and $4$ times faster than the rank-adaptive algorithm applied to the reduced system in the two cases, with comparable errors at the final time. We also observe that rank adaptivity yields lower errors at the final time compared to the non-rank-adaptive case in both the reduced and hyper-reduced systems, at the cost of a higher runtime. This interplay will be discussed further in the comment to \Cref{fig:localized_ra_comptimes}.

\begin{table}[h]
    {\footnotesize 
    \caption{Test 1, localized solution. Rank-adaptive case.
    Computational times and relative errors at the final time.
    Left: $2\Nrz=24$, $\tau_\nm=\tau_{\nms}=10^{-4}$. Right: $2\Nrz=40$, $\tau_\nm=\tau_{\nms}=10^{-6}$.}\label{tab:localized_ra}
    \begin{center}
    \begin{tabular}{|c|c|c|}\hline
        Model $2\Nrz=24$ & Runtime [s] & $\Ecal(T)$ \\ \hline
        ROM-RA & 1984.0 & 8.62e-04 \\ 
        hROM-RA & 116.4 & 9.90e-04 \\ \hline 
        ROM-NRA & 675.7 & 3.12e-03 \\ 
        hROM-NRA & 27.5 & 4.14e-03 \\ \hline 
    \end{tabular}\quad\;
        \begin{tabular}{|c|c|c|}\hline
        Model $2\Nrz=40$ & Runtime [s] & $\Ecal(T)$ \\ \hline
        ROM-RA & 1879.8 & 1.78e-04 \\ 
        hROM-RA & 469.9 & 1.94e-04 \\ \hline 
        ROM-NRA & 1360.7 & 6.50e-04 \\ 
        hROM-NRA & 38.0 & 1.56e-03 \\ \hline 
    \end{tabular}
    \end{center}
    }
\end{table}

Notice that the tolerances $\tau_\nm$ and $\tau_\nms$ and the EIM update frequency $\deimf$ have been kept fixed in these experiments. On the other hand, it is reasonable to assume that, as $\nm$ increases, one may afford larger values for these quantities and, hence, the computational gain in the hyper-reduced system might be improved when adaptive strategies to select $\deimf$ and $\tau_\nm$ are implemented. A rigorous development of this strategy is left for future work. Here we implement a preliminary test, where the value of $\tau_\nm$ is modified during the online phase. In particular, we fix $\tau_{\nms}=10^{-6}$ while we set $\tau_\nm=10^{-6}$ for $t\leq0.5$ and $\tau_\nm=10^{-4}$ for $t>0.5$. The threshold $t=0.5$ is determined \emph{a posteriori} from the behavior of the relative error over time displayed in \Cref{fig:localized_ra}.  With this choice, the computational time is $259.4\, s$, a reduction of almost $50\%$ compared to the case of fixed $\tau_\nm$, with a comparable relative error at the final time. This example shows the advantages of an adaptive strategy for the selection of the tolerance $\tau_\nm$.

To summarize, \Cref{fig:localized_ra_comptimes} shows the errors at the final time as a function of the computational time in the experiments discussed above.
\begin{figure}[ht]
\centering
\begin{tikzpicture}[every mark/.append style={mark size=2.5pt, line width=.8pt},scale=1]
    \begin{groupplot}[
      group style={group size=1 by 1,
                  horizontal sep=1.4cm},
      width=6.5cm, height=4.5cm
    ]
    \nextgroupplot[ylabel={$\Ecal(T)$},
                  xlabel={Runtime [s]},
                  xlabel style = {xshift=-.2cm},
                  ylabel style = {yshift=-.2cm},
                  axis line style = thick,
                  grid=both,
                  minor tick num=0,
                  xmode = log,
                  ymode = log,
                  grid style = {gray,opacity=0.2},
                  xmin=10^1, xmax=10^4,
                  ymin=10^(-4), ymax=10^(-2),
                  xlabel style={font=\footnotesize},
                  ylabel style={font=\footnotesize},
                  x tick label style={font=\footnotesize},
                  y tick label style={font=\footnotesize},
                  legend style={font=\footnotesize},
                  legend cell align={left},
                  legend columns = 1,
                  legend style={at={(1.6,1.08)},anchor=north}]
        \addplot+[only marks,color=black,mark=square] table[x=trnra12,y=ernra12] {localized_ra_comptimes.txt};
        \addplot+[only marks,color=black,mark=diamond] table[x=trra12,y=erra12] {localized_ra_comptimes.txt};
        \addplot+[only marks,color=black,mark=square*,mark options={solid,fill=black,line width=.8pt,mark size=2.5pt}] table[x=trnra20,y=ernra20] {localized_ra_comptimes.txt};
        \addplot+[only marks,color=black,mark=diamond*,mark options={solid,fill=black,line width=.8pt,mark size=2.5pt}] table[x=trra20,y=erra20] {localized_ra_comptimes.txt};
        \addplot+[only marks,color=red,mark=square] table[x=thrnra12,y=ehrnra12] {localized_ra_comptimes.txt};
        \addplot+[only marks,color=red,mark=diamond] table[x=thrra12,y=ehrra12] {localized_ra_comptimes.txt};
        \addplot+[only marks,color=red,mark=square*,mark options={solid,fill=red,line width=.8pt,mark size=2.5pt}] table[x=thrnra20,y=ehrnra20] {localized_ra_comptimes.txt};
        \addplot+[only marks,color=red,mark=diamond*,mark options={solid,fill=red,line width=.8pt,mark size=2.5pt}] table[x=thrra20,y=ehrra20] {localized_ra_comptimes.txt};
        \addplot+[only marks,color=yellow,mark=*,mark options={solid,fill=yellow,line width=.8pt,mark size=2.5pt}] table[x=tadapttaum,y=eadapttaum] {localized_ra_comptimes.txt};
        \legend{{ROM-NRA, $2\Nrz=24$},{ROM-RA, $2\Nrz=24$},{ROM-NRA, $2\Nrz=40$},{ROM-RA, $2\Nrz=40$},{hROM-NRA, $2\Nrz=24$},{hROM-RA, $2\Nrz=24$},{hROM-NRA, $2\Nrz=40$},{hROM-RA, $2\Nrz=40$},{hROM-RA, $2\Nrz=40$, variable $\tau_\nm$}};
    \end{groupplot}
\end{tikzpicture}
\caption{\footnotesize Test 1, localized solution. Relative errors at the final time vs. computational time.}\label{fig:localized_ra_comptimes}
\end{figure}
First, all red markers are to the left of the corresponding black markers meaning that hyper-reduction allows a reduced computational time to achieve comparable errors at the final time, both in the non-rank-adaptive and in the rank-adaptive case.
%
The increased runtime of the rank-adaptive algorithms is more evident in the hyper-reduced system when the dimension of the reduced space is large, because the EIM space dimension $\nm$ is also enlarged due to rank adaptation. Nevertheless, efficiency can be improved in these cases by an adaptive selection of the tolerance $\tau_\nm$ and/or of the update frequency $\deimf$ (circle mark in \Cref{fig:localized_ra_comptimes}).


For the reduced order model, we observe that the rank-adaptive reduced system with $2\Nrz=24$ (black empty diamond) is as computationally expensive as the non-rank-adaptive reduced system with $2\Nrz=40$ (black full square), with a slightly larger error at the final time. This is due to the fact that, as shown in \Cref{fig:localized_ra} (center), the dimension of the reduced space in the rank-adaptive case increases rapidly at the beginning of the simulation, reaching the value $\Nr=40$ at about $t=0.5$ in agreement with the rapid growth of the numerical rank of the full solution. Therefore, the reduced space in the rank-adaptive case is considerably smaller than in the non-adaptive case only in the first few time steps, which explains the similar runtimes of the two schemes. Moreover, the error in the rank-adaptive case is determined by the choice of the initial reduced dimension $2\Nrz$. Since the error in the non-adaptive case does not exhibit a significant growth at later stages of the simulation, its value at the final time is smaller than the one obtained with rank adaptivity and a smaller initial reduced space. 
Although this represents a possible case scenario, we expect, in general, the rank-adaptive (hyper-)reduced system to outperform its non-rank-adaptive counterpart with a larger initial reduced space in all situations where the rank of the full order solution undergoes significant changes over time.

For the case $2\Nrz=24$, we present in \Cref{fig:localized_ra_errind_n12} the time evolution of the error indicator $\indres$, and a comparison with the angle $\indtheta$ from \cite{CL23} and defined in \eqref{eq:theta_errind}. The latter fails to identify the increase in the rank of the solution: its value never satisfies the update criterion so that no rank update is performed. 

\begin{figure}[ht]
\centering
\begin{tikzpicture}
    \begin{groupplot}[
      group style={group size=2 by 1,
                  horizontal sep=1.5cm},
      width=6.5cm, height=4.5cm,
    ]
    \nextgroupplot[ylabel={$\indres(t)$},
                  xlabel={$t$},
                  xlabel style = {yshift=.2cm},
                  ylabel style = {yshift=-.2cm},
                  axis line style = thick,
                  grid=both,
                  minor tick num=0,
                  ytick={1e-05,1e-04,1e-03,1e-02,1e-01,1e+00},
                  ymode=log,
                  grid style = {gray,opacity=0.2},
                  xmin=0, xmax=3,
                  ymin=1e-05, ymax=1e+00,
                  xlabel style={font=\footnotesize},
                  ylabel style={font=\footnotesize},
                  x tick label style={font=\footnotesize},
                  y tick label style={font=\footnotesize},
                  legend style={font=\footnotesize},
                  legend cell align={left},
                  legend columns = 2,
                  legend style={at={(1.8,1.4)},anchor=north}]
        \addplot+[color=blue,mark=none,line width=1.5pt] table[x=t,y=res,each nth point={10}] {localized_ra_errind_n12.txt};
    \nextgroupplot[ylabel={$\indtheta(t)$},
                  xlabel={$t$},
                  xlabel style = {yshift=.2cm},
                  ylabel style = {yshift=-.2cm},
                  axis line style = thick,
                  grid=both,
                  minor tick num=0,
                  grid style = {gray,opacity=0.2},
                  ymode=log,
                  xmin=0, xmax=3,
                  ymin=1e-04, ymax=1e-02,
                  xlabel style={font=\footnotesize},
                  ylabel style={font=\footnotesize},
                  x tick label style={font=\footnotesize},
                  y tick label style={font=\footnotesize}]
        \addplot+[color=blue,mark=none,line width=1.5pt] table[x=t,y=theta,each nth point={10}] {localized_ra_errind_n12.txt};
    \end{groupplot}
\end{tikzpicture}
\caption{\footnotesize Test 1, localized solution.
Time evolution of the error indicators $\indres$ (left) and $\indtheta$ (right).}\label{fig:localized_ra_errind_n12}
\end{figure}

Finally, we report in \Cref{fig:localized_hamcons} the time evolution of the error \eqref{eq:err_hamcons} in the Hamiltonian conservation in the full order model and in the hyper-reduced system, with and without rank adaptivity. Although our method preserves the geometric structure of the full order system, we cannot guarantee exact preservation of the Hamiltonian, for several reasons. First, the Hamiltonian $\hrham$ associated to the hyper-reduced system is an approximation of the full order Hamiltonian $\Hcal$. Second, the time integrator employed in the numerical experiments is symplectic but does not preserve the Hamiltonian exactly at each time step. Third, the use of splitting techniques implies that, at each time step, the approximated solution is projected onto the space spanned by the updated basis, introducing an error in the conservation of invariants. Nevertheless, we achieve a good control on the Hamiltonian error, whose growth remains bounded over time.
\begin{figure}[ht]
\centering
\begin{tikzpicture}
    \begin{groupplot}[
      group style={group size=1 by 1,
                  horizontal sep=1.5cm},
      width=6.5cm, height=4.5cm,
    ]
    \nextgroupplot[ylabel={$\Ecal_\Hcal(t)$},
                  xlabel={$t$},
                  xlabel style = {yshift=.2cm},
                  ylabel style = {yshift=-.2cm},
                  axis line style = thick,
                  grid=both,
                  minor tick num=0,
                  ytick={1e-10,1e-09,1e-08,1e-07,1e-06,1e-05,1e-04,1e-03},
                  ymode=log,
                  grid style = {gray,opacity=0.2},
                  xmin=0, xmax=3,
                  ymin=1e-10, ymax=1e-03,
                  xlabel style={font=\footnotesize},
                  ylabel style={font=\footnotesize},
                  x tick label style={font=\footnotesize},
                  y tick label style={font=\footnotesize},
                  legend style={font=\footnotesize},
                  legend cell align={left},
                  legend columns = 1,
                  legend style={at={(1.6,0.94)},anchor=north}]
        \addplot+[color=ForestGreen,mark=none,dashed,line width=1.5pt] table[x=t,y=pUpA] {localized_hamcons.txt};
        \addplot+[color=Cyan,mark=none,dashed,line width=1.5pt] table[x=t,y=pUpAs] {localized_hamcons.txt};
        \addplot+[color=YellowOrange,mark=none,dashed,line width=1.5pt] table[x=t,y=pUspA] {localized_hamcons.txt};
        \addplot+[color=Red,mark=none,dashed,line width=1.5pt] table[x=t,y=pUspAs] {localized_hamcons.txt};
        \addplot+[color=Red,mark=none,solid,line width=1.5pt] table[x=t,y=ra] {localized_hamcons.txt};
        \addplot+[color=black,mark=none,dashed,line width=1.5pt] table[x=t,y=full] {localized_hamcons.txt};
        \legend{{hROM-NRA, $\npAs=\np,\npUs=\np$},{hROM-NRA, $\npAs<\np,\npUs=\np$},{hROM-NRA, $\npAs=\np,\npUs<\np$},{hROM-NRA, $\npAs<\np,\npUs<\np$},{hROM-RA, $\npAs<\np,\npUs<\np$},{FOM}};
    \end{groupplot}
\end{tikzpicture}
\caption{\footnotesize Test 1, localized solution. Evolution of the error \eqref{eq:err_hamcons} in the conservation of the Hamiltonian.}\label{fig:localized_hamcons}
\end{figure}

\subsection{Test 2: Non-localized solution}\label{sec:exp_nonlocalized}

In a second test, we fix $\prme=-1$ in equation \eqref{eq:schrodinger} and we consider the following initial condition \cite{HPR22}
\begin{equation*}
    u^0(\hc,\vc;\prmaic,\prmbic)=(1+\prmaic\sin{\hc})(2+\prmbic\sin{\vc}).
\end{equation*}
The parameter vector is $\prm=(\prmaic,\prmbic)$ and we take $\prms=[0.97,1.03]\times[0.97,1.03]\subset\Rbb^2$ as parameter set. We solve \eqref{eq:schrodinger} for $\np=100$ test parameters obtained by sampling $\prms$ with $10$ equispaced parameters in each direction. The full order system is solved in about $6258\,s$.
The numerical solution is shown in \Cref{fig:nonlocalized_fullsol_rankfullsol} for one test parameter. We observe that, unlike in the previous case, this problem is not characterized by coherent structures that are local in space, except for the initial stages of the simulation. 
Moreover, the numerical rank keeps increasing during the simulation (\Cref{fig:nonlocalized_fullsol_rankfullsol} right) which makes it a challenging test for our rank-adaptive strategy.
\begin{figure}[h]
\begin{minipage}{0.6\textwidth}
    \centering
    \includegraphics[width=\textwidth]{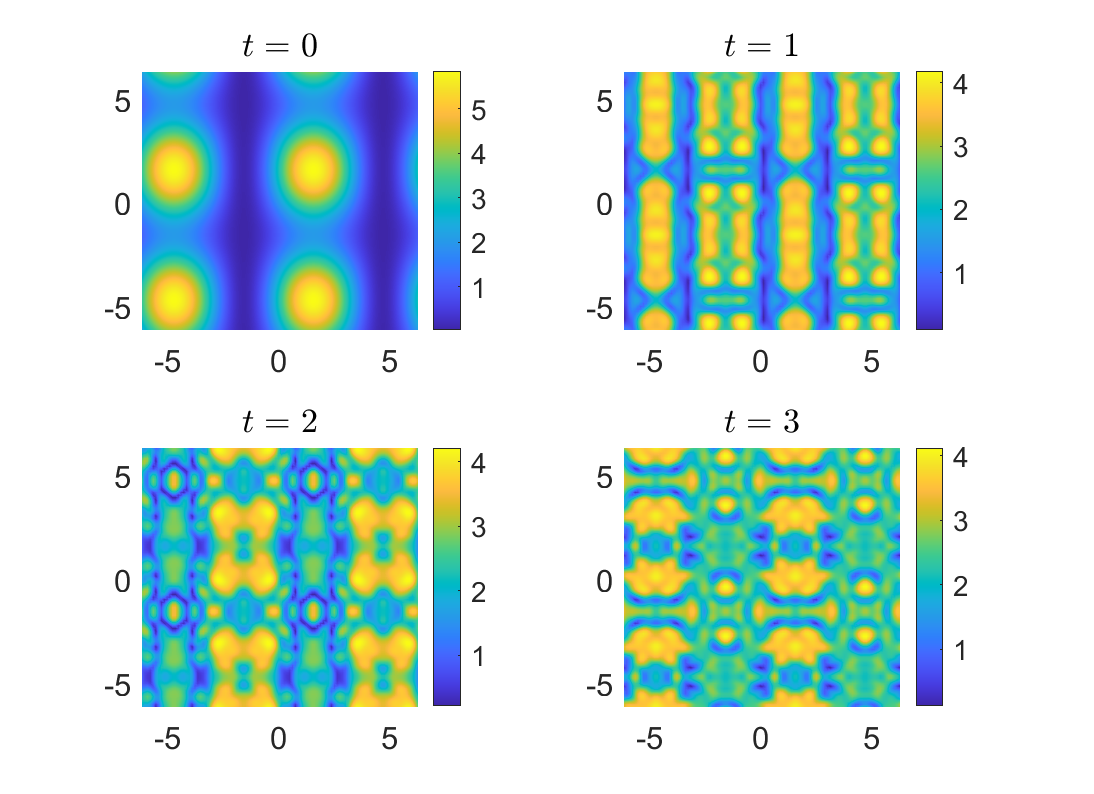}
\end{minipage}
\begin{minipage}{0.39\textwidth}
   \centering
\begin{tikzpicture}
    \begin{groupplot}[
      group style={group size=1 by 1,
                  horizontal sep=1.4cm},
      width=\textwidth, height=4.5cm
    ]
    \nextgroupplot[ylabel={$\epsilon$-rank},
                  xlabel={$t$},
                  xlabel style = {yshift=.2cm},
                  ylabel style = {yshift=-.5cm},
                  axis line style = thick,
                  grid=both,
                  minor tick num=0,
                  grid style = {gray,opacity=0.2},
                  xmin=0, xmax=3,
                  ymin=0, ymax=100,
                  xlabel style={font=\footnotesize},
                  ylabel style={font=\footnotesize},
                  x tick label style={font=\footnotesize},
                  y tick label style={font=\footnotesize},
                  legend style={font=\footnotesize},
                  legend cell align={left},
                  legend columns = 2,
                  legend style={at={(0.5,1.5)},anchor=north}]
        \addplot+[color=red,mark=none,line width=1.5pt] table[x=t,y=tau1e-1] {nonlocalized_rankfullsol_T.txt};
        \addplot+[color=green,mark=none,line width=1.5pt] table[x=t,y=tau1e-3] {nonlocalized_rankfullsol_T.txt};
        \addplot+[color=blue,mark=none,line width=1.5pt] table[x=t,y=tau1e-5] {nonlocalized_rankfullsol_T.txt};
        \addplot+[color=violet,mark=none,line width=1.5pt] table[x=t,y=tau1e-7] {nonlocalized_rankfullsol_T.txt};
        \addplot+[color=orange,mark=none,line width=1.5pt] table[x=t,y=tau1e-9] {nonlocalized_rankfullsol_T.txt};
        \legend{{$\epsilon=1e{-}01$}\;,{$\epsilon=1e{-}03$}\;,{$\epsilon=1e{-}05$}\;,{$\epsilon=1e{-}07$}\;,{$\epsilon=1e{-}09$}};
    \end{groupplot}
\end{tikzpicture}
  \end{minipage}
\caption{\footnotesize Test 2, non-localized solution. Left: modulus of the full order solution at four different time instants for $\prm=(0.97,0.97)$. Right: $\epsilon$-rank of the full solution $\fsR(t)\in\Rbb^{20000\times 100}$.}\label{fig:nonlocalized_fullsol_rankfullsol}
\end{figure}

\subsubsection{Non-rank-adaptive algorithm}
First, we assess the performances of the reduced and hyper-reduced systems when the dimension of the reduced space is not updated during the simulation. To this end, we fix $\Nr=2\Nrz=8$ which is the exact rank of the initial condition.
The evolution of the relative errors with respect to the full order solution is shown in \Cref{fig:nonlocalized_nonra_n4_m25} (left). Differently from the first test case, numerical errors increase at a constant rate until the final time, which is in accordance with the increase of the numerical rank of the full solution from \Cref{fig:nonlocalized_fullsol_rankfullsol}. In the hyper-reduced system, the initial EIM basis is constructed via POD of the matrix of the nonlinear reduced Jacobian with tolerance $\tau_\nm=10^{-10}$, which yields $\nm_0=25$. The EIM update is then performed at every time step with $\tau_\nms=10^{-4}$. We test different combinations of the choices of sample parameters $\npAs$ and $\npUs$. The computational time required by the different algorithms is reported in \Cref{fig:nonlocalized_times_err_ms} (left). 

In the best case scenario, we observe a reduction of the computational time by a factor $25$ compared to the reduced order model. An important remark is that, in contrast with the localized case, updating the EIM basis using $\npUs<\np$ sample parameters does not yield any computational gain; on the contrary, this choice turns out to be more expensive than performing the update with all $\np$ test parameters. The reason behind this behavior is related to the number $\nms$ of sample EIM indices, which is shown in \Cref{fig:nonlocalized_times_err_ms} (right). The non-local nature of the problem causes the value of $\nms$ to be comparable to $\nd=\Nfh$ for a large portion of the time window $[0,T]$. 
As shown in \Cref{sec:adaptEIM}, the arithmetic complexity of selecting $\npUs$ sample indices and updating the EIM basis depends, among others, on the number of sample indices $\nms$, and this cost becomes dominant when $\nms\approx\nd$. 
Using only $\npUs$ sample parameters to perform the EIM update does not compensate for this extra cost, even when $\npUs\ll\np$. In view of this fact, we will set $\npUs=\np$ in the following. On the other hand, a significant reduction of computational time can be seen when $\npAs<\np$ sample parameters are used to evolve the reduced basis. This reduction amounts to over $30\%$ when $\npUs=\np$. In all cases, there is little difference in the relative errors compared to the benchmark case $\npUs=\npAs=\np$.

\begin{figure}[ht]
\centering
\begin{tikzpicture}
    \begin{groupplot}[
      group style={group size=2 by 1,
                  horizontal sep=1.4cm},
      width=6.5cm, height=5cm
    ]
    \nextgroupplot[ylabel={$\mathcal{E}(t)$},
                  xlabel={$t$},
                  xlabel style = {yshift=.2cm},
                  ylabel style = {yshift=-.2cm},
                  axis line style = thick,
                  grid=both,
                  minor tick num=0,
                  ytick={1e-01,1e-02,1e-03,1e-04,1e-05,1e-06,1e-07},
                  grid style = {gray,opacity=0.2},
                  xmin=0, xmax=3,
                  ymin=1e-07, ymax=1e-01,
                  ymode=log,
                  xlabel style={font=\footnotesize},
                  ylabel style={font=\footnotesize},
                  x tick label style={font=\footnotesize},
                  y tick label style={font=\footnotesize},
                  legend style={font=\footnotesize},
                  legend cell align={left},
                  legend columns = 1,
                  legend style={at={(0.58,0.64)},anchor=north}]
        \addplot+[color=black,dashed,mark=none,line width=1.5pt] table[x=t,y=red] {nonlocalized_nonra_errors.txt};
        \addplot+[color=ForestGreen,mark=none,line width=1.5pt] table[x=t,y=hred_pApU] {nonlocalized_nonra_errors.txt};
        \addplot+[color=Cyan,mark=none,line width=1.5pt] table[x=t,y=hred_pApUs] {nonlocalized_nonra_errors.txt};
        \addplot+[color=YellowOrange,mark=none,line width=1.5pt] table[x=t,y=hred_pAspU] {nonlocalized_nonra_errors.txt};
        \addplot+[color=red,mark=none,line width=1.5pt] table[x=t,y=hred_pAspUs] {nonlocalized_nonra_errors.txt};
        \legend{{ROM},{hROM, $\npAs=\np,\npUs=\np$},{hROM, $\npAs<\np,\npUs=\np$},{hROM, $\npAs=\np,\npUs<\np$},{hROM, $\npAs<\np,\npUs<\np$}};
    \nextgroupplot[ylabel={$\npUs,\npAs$},
                  xlabel={$t$},
                  xlabel style = {yshift=.2cm},
                  ylabel style = {yshift=-.4cm},
                  axis line style = thick,
                  grid=both,
                  minor tick num=0,
                  grid style = {gray,opacity=0.2},
                  xmin=0, xmax=3,
                  ymin=0, ymax=10,
                  xlabel style={font=\footnotesize},
                  ylabel style={font=\footnotesize},
                  x tick label style={font=\footnotesize},
                  y tick label style={font=\footnotesize}]
        \addplot+[color=Cyan,dashed,mark=none,line width=1.5pt] table[x=tpAs_pAspU,y=pAs_pAspU] {nonlocalized_nonra_pAstar_pUstar.txt};
        \addplot+[color=YellowOrange,mark=none,line width=1.5pt] table[x=tpUs_pApUs,y=pUs_pApUs] {nonlocalized_nonra_pAstar_pUstar.txt};
        \addplot+[color=red,mark=none,line width=1.5pt] table[x=tpUs_pAspUs,y=pUs_pAspUs] {nonlocalized_nonra_pAstar_pUstar.txt};
        \addplot+[color=red,dashed,dash phase=3pt,mark=none,line width=1.5pt] table[x=tpAs_pAspUs,y=pAs_pAspUs] {nonlocalized_nonra_pAstar_pUstar.txt};
    \end{groupplot}
\end{tikzpicture}
\caption{\footnotesize Test 2, non-localized solution. Dimension of the dynamical reduced basis space $\Nr=2\Nrz=8$. Left: evolution of the relative errors of the reduced and hyper-reduced systems for $\nm_\tind=\nm_0=25$ and with different choices of the sample parameters. Right: evolution of $\npUs$ (solid lines) and of $\npAs$ (dashed lines).}\label{fig:nonlocalized_nonra_n4_m25}
\end{figure}

\begin{figure}[h]
\begin{minipage}{0.5\linewidth}
{\footnotesize
    \begin{center}
    \begin{tabular}{|c|c|c|}\hline
        Model & Time [s] & $\Ecal(T)$ \\ \hline
        ROM & 375.7 & 4.61e{-}02 \\ \hline
        hROM $\npAs=\np,\npUs=\np$ & 21.1 & 4.73e-02 \\
        hROM $\npAs<\np,\npUs=\np$ & 14.8 & 5.99e-02 \\
        hROM $\npAs=\np,\npUs<\np$ & 29.0 & 6.09e-02 \\
        hROM $\npAs<\np,\npUs<\np$ & 23.2 & 7.43e-02 \\ \hline
    \end{tabular}
    \end{center}
}
\end{minipage}\hfill
\begin{minipage}{0.45\linewidth}
   \centering
        \begin{tikzpicture}
        \begin{groupplot}[
          group style={group size=1 by 1,
                      horizontal sep=1.4cm},
          width=5.5cm, height=4.5cm
        ]
        \nextgroupplot[ylabel={$\nms$},
                      xlabel={$t$},
                      xlabel style = {yshift=.2cm},
                      ylabel style = {yshift=-.2cm},
                      axis line style = thick,
                      grid=both,
                      minor tick num=0,
                      grid style = {gray,opacity=0.2},
                      xmin=0, xmax=3,
                      ymin=0, ymax=10000,
                      xlabel style={font=\footnotesize},
                      ylabel style={font=\footnotesize},
                      x tick label style={font=\footnotesize},
                      y tick label style={font=\footnotesize},
                      legend style={font=\footnotesize},
                      legend cell align={left},
                      legend columns = 1,
                      legend style={at={(0.6,0.65)},anchor=north}]
            \addplot+[color=ForestGreen,mark=none,each nth point={10},line width=0.5pt] table[x=t,y=pApU] {nonlocalized_nonra_ms.txt};
            \addplot+[color=Cyan,mark=none,each nth point={10},line width=0.5pt] table[x=t,y=pAspU] {nonlocalized_nonra_ms.txt};
            \addplot+[color=YellowOrange,mark=none,each nth point={10},line width=0.5pt] table[x=t,y=pApUs] {nonlocalized_nonra_ms.txt};
            \addplot+[color=red,mark=none,each nth point={10},line width=0.5pt] table[x=t,y=pAspUs] {nonlocalized_nonra_ms.txt};
            \legend{{$\npAs=\np,\npUs=\np$},{$\npAs<\np,\npUs=\np$},{$\npAs=\np,\npUs<\np$},{$\npAs<\np,\npUs<\np$}};
        \end{groupplot}
    \end{tikzpicture}
  \end{minipage}
\caption{Test 2, non-localized solution. Non-rank-adaptive case, $\Nr=2\Nrz=8$, $\nm_\tind=\nm_0=25$. Left: computational times and relative errors at the final time. Right: evolution of the number $\nms$ of EIM sample indices selected according to \Cref{alg:DEIM_update} with $\tau_{\nms}=10^{-4}$.}\label{fig:nonlocalized_times_err_ms}
\end{figure}

\subsubsection{Rank-adaptive algorithm}
We consider $\radapt=\cadapt=1.1$ in the update criterion, and we compare the two error indicators described in \Cref{sec:rank-adaptivity}. The dimension of the reduced space at the initial time is $2\Nrz=8$ and a tolerance $\tau_\nm=\{10^{-6},10^{-8}\}$ is used to adapt the dimension of the EIM space. The EIM sample indices are computed based on the tolerance $\tau_{\nms}=10^{-4}$. On the basis of the previous experiments, we set $\npUs=\np$ and $\npAs<\np$ in the hyper-reduced system. \Cref{fig:nonlocalized_ra_n4} shows the time evolution of the relative error in the reduced and hyper-reduced systems (left), the dimension $\Nr$ of the reduced space (center) and the dimension $\nm_\tind$ of the EIM space (right). The results are compared to the error obtained in the reduced system with the same value of $2\Nrz$ and no rank adaptivity.

\begin{figure}[h]
\centering
\begin{tikzpicture}
    \begin{groupplot}[
      group style={group size=3 by 1,
                  horizontal sep=1.4cm},
      width=4.5cm, height=4.5cm,
    ]
    \nextgroupplot[ylabel={$\Ecal(t)$},
                  xlabel={$t$},
                  xlabel style = {yshift=.2cm},
                  ylabel style = {yshift=-.2cm},
                  axis line style = thick,
                  grid=both,
                  minor tick num=0,
                  ymode=log,
                  grid style = {gray,opacity=0.2},
                  xmin=0, xmax=3,
                  ymin=1e-06, ymax=1e-01,
                  ytick={1e-06,1e-05,1e-04,1e-03,1e-02,1e-01},
                  xlabel style={font=\footnotesize},
                  ylabel style={font=\footnotesize},
                  x tick label style={font=\footnotesize},
                  y tick label style={font=\footnotesize},
                  legend style={font=\footnotesize},
                  legend cell align={left},
                  legend columns = 3,
                  legend style={at={(1.95,1.4)},anchor=north}]
        \addplot+[color=black,solid,mark=none,line width=1.25pt] table[x=t,y=red] {nonlocalized_nonra_errors.txt};
        \addplot+[color=NavyBlue,solid,mark=none,line width=1.25pt] table[x=t,y=pred] {nonlocalized_ra_errors.txt};
        \addplot+[color=Cyan,dashdotted,mark=none,line width=1.5pt] table[x=t,y=p6] {nonlocalized_ra_errors.txt};
        \addplot+[color=Cyan,dotted,mark=none,line width=1.5pt] table[x=t,y=p8] {nonlocalized_ra_errors.txt};
        \addplot+[color=Bittersweet,solid,mark=none,line width=1.25pt] table[x=t,y=tred] {nonlocalized_ra_errors.txt};
        \addplot+[color=orange,dashdotted,mark=none,line width=1.5pt] table[x=t,y=t6] {nonlocalized_ra_errors.txt};
        \addplot+[color=orange,dotted,mark=none,line width=1.5pt] table[x=t,y=t8] {nonlocalized_ra_errors.txt};
        \legend{{},{ROM-RA, $\indres$}\;,{hROM-RA, $\indres$, $\tau_\nm=10^{-6}$}\;,{hROM-RA, $\indres$, $\tau_\nm=10^{-8}$},{ROM-RA, $\indtheta$}\;,{hROM-RA, $\indtheta$, $\tau_\nm=10^{-6}$}\;,{hROM-RA, $\indtheta$, $\tau_\nm=10^{-8}$}};
    \nextgroupplot[ylabel={$\Nr$},
                  xlabel={$t$},
                  xlabel style = {yshift=.2cm},
                  ylabel style = {yshift=-.4cm},
                  axis line style = thick,
                  grid=both,
                  minor tick num=0,
                  grid style = {gray,opacity=0.2},
                  xmin=0, xmax=3,
                  ymin=0, ymax=34,
                  xlabel style={font=\footnotesize},
                  ylabel style={font=\footnotesize},
                  x tick label style={font=\footnotesize},
                  y tick label style={font=\footnotesize}]
        \addplot+[color=black,solid,mark=none,line width=1.25pt] table[x=tnared,y=nared] {nonlocalized_ra_2n.txt};
        \addplot+[color=NavyBlue,solid,mark=none,line width=1.25pt] table[x=tpred,y=pred] {nonlocalized_ra_2n.txt};
        \addplot+[color=Cyan,dashdotted,mark=none,line width=1.5pt] table[x=tp6,y=p6] {nonlocalized_ra_2n.txt};
        \addplot+[color=Cyan,dotted,mark=none,line width=1.5pt] table[x=tp8,y=p8] {nonlocalized_ra_2n.txt};
        \addplot+[color=Bittersweet,solid,mark=none,line width=1.25pt] table[x=ttred,y=tred] {nonlocalized_ra_2n.txt};
        \addplot+[color=orange,dashdotted,mark=none,line width=1.5pt] table[x=tt6,y=t6] {nonlocalized_ra_2n.txt};
        \addplot+[color=orange,dotted,mark=none,line width=1.5pt] table[x=tt8,y=t8] {nonlocalized_ra_2n.txt};
    \nextgroupplot[ylabel={$\nm_\tind$},
                  xlabel={$t$},
                  xlabel style = {yshift=.2cm},
                  ylabel style = {yshift=-.25cm},
                  axis line style = thick,
                  grid=both,
                  minor tick num=0,
                  grid style = {gray,opacity=0.2},
                  xmin=0, xmax=3,
                  ymin=0, ymax=700,
                  xlabel style={font=\footnotesize},
                  ylabel style={font=\footnotesize},
                  x tick label style={font=\footnotesize},
                  y tick label style={font=\footnotesize}]
        \addplot+[color=Cyan,dashdotted,mark=none,line width=1.5pt] table[x=tp6,y=p6] {nonlocalized_ra_m.txt};
        \addplot+[color=Cyan,dotted,mark=none,line width=1.5pt] table[x=tp8,y=p8] {nonlocalized_ra_m.txt};
        \addplot+[color=orange,dashdotted,mark=none,line width=1.5pt] table[x=tt6,y=t6] {nonlocalized_ra_m.txt};
        \addplot+[color=orange,dotted,mark=none,line width=1.5pt] table[x=tt8,y=t8] {nonlocalized_ra_m.txt};
    \end{groupplot}
\end{tikzpicture}
\caption{\footnotesize Test 2, non-localized solution. Rank-adaptive case, $2\Nrz=8$, $\tau_{\nms}=10^{-4}$, different error indicators. Left: evolution of the relative errors of the reduced and hyper-reduced systems. The black solid line refers to the non-rank-adaptive reduced order model (ROM-NRA). Center: evolution of the dimension of the reduced space. Right: evolution of the dimension of the EIM~space.}\label{fig:nonlocalized_ra_n4}
\end{figure}

When $\indres$ is used as error indicator and $\tau_{\nm}=10^{-8}$, the error of the rank-adaptive hyper-reduction algorithm is comparable to that of the reduced system. Moreover, the errors resulting from the rank-adaptive algorithms increase at a smaller rate than in the non-adaptive method. One may be interested in exploiting the rank-adaptive strategy to keep the error approximately constant over time. One possibility to achieve this is to set $\cadapt=1$, so that the update frequency is not limited. Nevertheless, one should always keep in mind the interplay between accuracy and efficiency. Numerical tests not reported here show that the reduced space might become unnecessarily large without any control on the update frequency. In addition, other factors, such as the error due to time integration, prevent the errors to remain flat.

When $\indtheta$ is used as error indicator, the updates of the reduced dimension are not frequent enough to produce a significant improvement in the numerical errors. Our indicator $\indres$ allows for a better control: the updates are more frequent, especially at the beginning of the simulation, and the reduced space is larger at the final time. 

Next we present a brief discussion of the computational efficiency of the schemes described above. A plot of $\Ecal(T)$ vs. the computational time is shown in \Cref{fig:nonlocalized_ra_comptimes}. We observe that the hyper-reduced model always outperforms its reduced order counterpart as it yields comparable errors at smaller algorithm runtimes. This computational gain is less substantial in the rank-adaptive case because the dimension of the EIM space $\nm_\tind$ must also be increased. As mentioned before, this can be addressed by adaptive techniques to select the update frequency $\deimf$ and the tolerance $\tau_\nm$. Finally, using $\indres$ rather than $\indtheta$ as an error indicator proves to be a better choice as the errors are smaller with only a slight increase of the computational time.

\begin{figure}[ht]
\centering
\begin{tikzpicture}[every mark/.append style={mark size=2.5pt, line width=.8pt},scale=1]
    \begin{groupplot}[
      group style={group size=1 by 1,
                  horizontal sep=1.4cm},
      width=6.5cm, height=4.5cm
    ]
    \nextgroupplot[ylabel={$\Ecal(T)$},
                  xlabel={Runtime [s]},
                  ylabel style = {yshift=-.2cm},
                  axis line style = thick,
                  grid=both,
                  minor tick num=0,
                  xmode = log,
                  ymode = log,
                  grid style = {gray,opacity=0.2},
                  xmin=10^1, xmax=10^3,
                  ymin=10^(-3), ymax=10^(-1),
                  xlabel style={font=\footnotesize},
                  ylabel style={font=\footnotesize},
                  x tick label style={font=\footnotesize},
                  y tick label style={font=\footnotesize},
                  legend style={font=\footnotesize},
                  legend cell align={left},
                  legend columns = 1,
                  legend style={at={(1.5,1.07)},anchor=north}]
        \addplot+[only marks,color=black,mark=square] table[x=tred_na,y=ered_na] {nonlocalized_ra_comptimes.txt};
        \addplot+[only marks,color=black,mark=triangle] table[x=tred_at,y=ered_at] {nonlocalized_ra_comptimes.txt};
        \addplot+[only marks,color=black,mark=diamond] table[x=tred_ap,y=ered_ap] {nonlocalized_ra_comptimes.txt};
        \addplot+[only marks,color=Cyan,mark=square] table[x=thred_na,y=ehred_na] {nonlocalized_ra_comptimes.txt};
        \addplot+[only marks,color=orange,mark=triangle] table[x=thred_at6,y=ehred_at6] {nonlocalized_ra_comptimes.txt};
        \addplot+[only marks,color=Cyan,mark=diamond] table[x=thred_ap6,y=ehred_ap6] {nonlocalized_ra_comptimes.txt};
        \addplot+[only marks,color=orange,mark=triangle*,mark options={solid,fill=orange,line width=.8pt,mark size=2.5pt}] table[x=thred_at8,y=ehred_at8] {nonlocalized_ra_comptimes.txt};
        \addplot+[only marks,color=Cyan,mark=diamond*,mark options={solid,fill=Cyan,line width=.8pt,mark size=2.5pt}] table[x=thred_ap8,y=ehred_ap8] {nonlocalized_ra_comptimes.txt};
        \legend{{ROM-NRA},{ROM-RA, $\indtheta$},{ROM-RA, $\indres$},{hROM-NRA},{hROM-RA, $\indtheta$, $\tau_\nm=10^{-6}$},{hROM-RA, $\indres$, $\tau_\nm=10^{-6}$},{hROM-RA, $\indtheta$, $\tau_\nm=10^{-8}$},{hROM-RA, $\indres$, $\tau_\nm=10^{-8}$}};
    \end{groupplot}
\end{tikzpicture}
\caption{\footnotesize Test 2, non-localized solution. Relative errors at the final time vs. computational time.}\label{fig:nonlocalized_ra_comptimes}
\end{figure}

We conclude this section by analyzing the time evolution of the Hamiltonian error \eqref{eq:err_hamcons}, shown in \Cref{fig:nonlocalized_hamcons}. Similar conclusions to the case of a localized solution hold here. In particular, we observe that the growth of the error can be controlled by adapting the dimension of the reduced space.
\begin{figure}[ht]
\centering
\begin{tikzpicture}
    \begin{groupplot}[
      group style={group size=1 by 1,
                  horizontal sep=1.5cm},
      width=6.5cm, height=4.5cm,
    ]
    \nextgroupplot[ylabel={$\Ecal_\Hcal(t)$},
                  xlabel={$t$},
                  xlabel style = {yshift=.2cm},
                  ylabel style = {yshift=-.2cm},
                  axis line style = thick,
                  grid=both,
                  minor tick num=0,
                  ytick={1e-08,1e-07,1e-06,1e-05,1e-04,1e-03,1e-02,1e-01},
                  ymode=log,
                  grid style = {gray,opacity=0.2},
                  xmin=0, xmax=3,
                  ymin=1e-08, ymax=1e-01,
                  xlabel style={font=\footnotesize},
                  ylabel style={font=\footnotesize},
                  x tick label style={font=\footnotesize},
                  y tick label style={font=\footnotesize},
                  legend style={font=\footnotesize},
                  legend cell align={left},
                  legend columns = 1,
                  legend style={at={(1.6,0.94)},anchor=north}]
        \addplot+[color=ForestGreen,mark=none,dashed,line width=1.5pt] table[x=t,y=pUpA] {nonlocalized_hamcons.txt};
        \addplot+[color=Cyan,mark=none,dashed,line width=1.5pt] table[x=t,y=pUpAs] {nonlocalized_hamcons.txt};
        \addplot+[color=YellowOrange,mark=none,dashed,line width=1.5pt] table[x=t,y=pUspA] {nonlocalized_hamcons.txt};
        \addplot+[color=Red,mark=none,dashed,line width=1.5pt] table[x=t,y=pUspAs] {nonlocalized_hamcons.txt};
        \addplot+[color=Cyan,mark=none,solid,line width=1.5pt] table[x=t,y=ra] {nonlocalized_hamcons.txt};
        \addplot+[color=black,mark=none,dashed,line width=1.5pt] table[x=t,y=full] {nonlocalized_hamcons.txt};
        \legend{{hROM-NRA, $\npAs=\np,\npUs=\np$},{hROM-NRA, $\npAs<\np,\npUs=\np$},{hROM-NRA, $\npAs=\np,\npUs<\np$},{hROM-NRA, $\npAs<\np,\npUs<\np$},{hROM-RA, $\npAs<\np,\npUs=\np$},{FOM}};
    \end{groupplot}
\end{tikzpicture}
\caption{\footnotesize Test 2, non-localized solution. Time evolution of the error \eqref{eq:err_hamcons} in the conservation of the Hamiltonian.}\label{fig:nonlocalized_hamcons}
\end{figure}


\section{Concluding remarks}
\label{sec:conclusions}

We have developed adaptive hyper-reduced models for Hamiltonian systems that need to be tested for a high number of parameters.
The proposed algorithm combines a symplectic dynamical low-rank approximation to approximate the state with an adaptive hyper-reduction strategy to efficiently deal with nonlinear Hamiltonian vector fields. The resulting hyper-reduced models preserve the geometric structure of the Hamiltonian flow and can be solved at a cost that is linear in the dimension $\Nfh$ of the full order model, linear in the number $\np$ of test parameters and that does not depend on the product $\Nfh\np$. Moreover, the adaptivity of the approximate models allows to deal with transport-dominated problems and conservative dynamical systems characterized by slowly decaying Kolmogorov $n$-widths.
The further adaptivity of both the dimension of the reduced basis space and of the EIM hyper-reduction space allows to achieve the desired accuracy in situations when the (numerical) rank of the solution has high variations over time.
The adaptivity is driven by a novel error indicator which, to the best of our knowledge, outperforms the error indicators currently available in the literature, in terms of numerical errors and computational complexity.
Finally, we have proposed a strategy to select parameter subsamples for the adaptive components of the algorithm to improve computational efficiency while ensuring favorable approximability properties. 

The performances of the proposed rank-adaptive hyper-reduction could be further improved by developing adaptive strategies to select the tolerance $\tau_\nm$, which determines how the dimension of the EIM space changes in time, and the frequency $\deimf$ at which the EIM space is updated.

\bibliographystyle{plain}

\bibliography{references}

\begin{thebibliography}{10}

\bibitem{AH17}
{\sc B.~M. Afkham and J.~S. Hesthaven}, {\em Structure preserving model
  reduction of parametric {H}amiltonian systems}, SIAM J. Sci. Comput., 39
  (2017), pp.~A2616--A2644.

\bibitem{BMNP04}
{\sc M.~Barrault, Y.~Maday, N.~C. Nguyen, and A.~T. Patera}, {\em An `empirical
  interpolation' method: application to efficient reduced-basis discretization
  of partial differential equations}, C. R. Math. Acad. Sci. Paris, 339 (2004),
  pp.~667--672.

\bibitem{BGW15}
{\sc P.~Benner, S.~Gugercin, and K.~Willcox}, {\em A survey of projection-based
  model reduction methods for parametric dynamical systems}, SIAM Review, 57
  (2015), pp.~483--531.

\bibitem{CdS01}
{\sc A.~Cannas~da Silva}, {\em Lectures on symplectic geometry}, vol.~1764 of
  Lecture Notes in Mathematics, Springer-Verlag, Berlin, 2001.

\bibitem{CKL22}
{\sc G.~Ceruti, J.~Kusch, and C.~Lubich}, {\em A rank-adaptive robust
  integrator for dynamical low-rank approximation}, BIT, 62 (2022),
  pp.~1149--1174.

\bibitem{CL23}
{\sc A.~Charous and P.~F.~J. Lermusiaux}, {\em Stable rank-adaptive dynamically
  orthogonal runge–kutta schemes}, SIAM Journal on Scientific Computing, 46
  (2024), pp.~A529--A560.

\bibitem{CS10}
{\sc S.~Chaturantabut and D.~C. Sorensen}, {\em Nonlinear model reduction via
  discrete empirical interpolation}, SIAM J. Sci. Comput., 32 (2010),
  pp.~2737--2764.

\bibitem{deVore17}
{\sc R.~A. DeVore}, {\em The theoretical foundation of reduced basis methods},
  Society for Industrial and Applied Mathematics, 2017, ch.~3.

\bibitem{DG16}
{\sc Z.~Drma{\v{c}} and S.~Gugercin}, {\em A new selection operator for the
  discrete empirical interpolation method---improved a priori error bound and
  extensions}, SIAM J. Sci. Comput., 38 (2016), pp.~A631--A648.

\bibitem{EL17}
{\sc V.~Ehrlacher and D.~Lombardi}, {\em A dynamical adaptive tensor method for
  the {V}lasov–{P}oisson system}, Journal of Computational Physics, 339
  (2017), pp.~285--306.

\bibitem{Fren34}
{\sc J.~Frenkel}, {\em Wave mechanics, advanced general theory}, Clarendon
  Press, Oxford, 1934.

\bibitem{GA22}
{\sc B.~Gao and P.-A. Absil}, {\em A {R}iemannian rank-adaptive method for
  low-rank matrix completion}, Comput. Optim. Appl., 81 (2022), pp.~67--90.

\bibitem{Gro19}
{\sc T.~H. Gronwall}, {\em Note on the derivatives with respect to a parameter
  of the solutions of a system of differential equations}, Ann. of Math. (2),
  20 (1919), pp.~292--296.

\bibitem{HLW06}
{\sc E.~Hairer, C.~Lubich, and G.~Wanner}, {\em Geometric numerical
  integration}, vol.~31 of Springer Series in Computational Mathematics,
  Springer-Verlag, Berlin, second~ed., 2006.

\bibitem{HP20}
{\sc J.~S. Hesthaven and C.~Pagliantini}, {\em Structure-preserving reduced
  basis methods for {P}oisson systems}, Math. Comp., 90 (2021), pp.~1701--1740.

\bibitem{HPR22}
{\sc J.~S. Hesthaven, C.~Pagliantini, and N.~Ripamonti}, {\em Rank-adaptive
  structure-preserving model order reduction of {H}amiltonian systems}, ESAIM
  Math. Model. Numer. Anal., 56 (2022), pp.~617--650.

\bibitem{HPRo22}
{\sc J.~S. Hesthaven, C.~Pagliantini, and G.~Rozza}, {\em Reduced basis methods
  for time-dependent problems}, Acta Numerica, 31 (2022), p.~265–345.

\bibitem{HNS23}
{\sc M.~Hochbruck, M.~Neher, and S.~Schrammer}, {\em Rank-adaptive dynamical
  low-rank integrators for first-order and second-order matrix differential
  equations}, BIT, 63 (2023), pp.~Paper No. 9, 24.

\bibitem{KL07}
{\sc O.~Koch and C.~Lubich}, {\em Dynamical low-rank approximation}, SIAM J.
  Matrix Anal. Appl., 29 (2007), pp.~434--454.

\bibitem{MR99}
{\sc J.~E. Marsden and T.~S. Ratiu}, {\em Introduction to mechanics and
  symmetry}, vol.~17 of Texts in Applied Mathematics, Springer-Verlag, New
  York, second~ed., 1999.

\bibitem{ML64}
{\sc A.~D. McLachlan}, {\em A variational solution of the time-dependent
  {S}chr\"{o}dinger equation}, Molecular Phys., 8 (1964), pp.~39--44.

\bibitem{MN17}
{\sc E.~Musharbash, F.~Nobile, and E.~Vidli\v{c}kov\'{a}}, {\em Symplectic
  dynamical low rank approximation of wave equations with random parameters},
  BIT, 60 (2020), pp.~1153--1201.

\bibitem{NB23}
{\sc M.~H. Naderi and H.~Babaee}, {\em Adaptive sparse interpolation for
  accelerating nonlinear stochastic reduced-order modeling with time-dependent
  bases}, Computer Methods in Applied Mechanics and Engineering, 405 (2023),
  p.~115813.

\bibitem{P19}
{\sc C.~Pagliantini}, {\em Dynamical reduced basis methods for {H}amiltonian
  systems}, Numer. Math., 148 (2021), pp.~409--448.

\bibitem{PV22}
{\sc C.~Pagliantini and F.~Vismara}, {\em Gradient-preserving hyper-reduction
  of nonlinear dynamical systems via discrete empirical interpolation}, SIAM J.
  Sci. Comput.,  (2023).

\bibitem{Peher20}
{\sc B.~Peherstorfer}, {\em Model reduction for transport-dominated problems
  via online adaptive bases and adaptive sampling}, SIAM J. Sci. Comput., 42
  (2020), pp.~A2803--A2836.

\bibitem{Peher15}
{\sc B.~Peherstorfer and K.~Willcox}, {\em Online adaptive model reduction for
  nonlinear systems via low-rank updates}, SIAM J. Sci. Comput., 37 (2015),
  pp.~A2123--A2150.

\bibitem{PM16}
{\sc L.~Peng and K.~Mohseni}, {\em Symplectic model reduction of {H}amiltonian
  systems}, SIAM J. Sci. Comput., 38 (2016), pp.~A1--A27.

\end{thebibliography}

\end{document}